\documentclass[11pt]{article}

\usepackage[margin=1in]{geometry}
\usepackage{amsmath,amssymb,amsthm,mathtools}
\usepackage{enumitem}
\usepackage{hyperref}

\newtheorem{theorem}{Theorem}
\newtheorem{lemma}[theorem]{Lemma}
\newtheorem{proposition}[theorem]{Proposition}
\newtheorem{corollary}[theorem]{Corollary}
\theoremstyle{definition}

\theoremstyle{remark}
\newtheorem{remark}[theorem]{Remark}
\theoremstyle{conjecture}
\newtheorem{conjecture}[theorem]{Conjecture}

\newcommand{\Z}{\mathbb{Z}}
\newcommand{\abs}[1]{\left|#1\right|}

\title{Critical Numbers for Restricted Sumsets:
Rigidity and Collapse in Finite Abelian Groups}

\author{Bocong Chen$^1$ and Jing Huang$^2$\footnote{E-mail addresses: {\it
mabcchen@scut.edu.cn (B. Chen),
jhuangmath@gzhu.edu.cn (J. Huang).
Corresponding author: Jing Huang.
}}
}

\date{\small
$1.$ School of Mathematics, South China University of Technology, Guangzhou 510641, China\\
$2.$
School of Mathematics and Information Science,
Guangzhou University, Guangzhou 510006, China\\
}

\begin{document}
\maketitle

\begin{abstract}
This paper establishes a classification of the
critical numbers for restricted sumsets in
finite abelian groups, determining them
exactly for even-order groups and bounding them
for odd-order groups, while revealing a
fundamental structural dichotomy governed by parity.
For groups of even order,
we prove a universal rigidity theorem:
the index-$2$ subgroup creates an immutable
arithmetic barrier at density $1/2$,
fixing the critical number at
$|G|/2+1$ regardless of the group's
internal structure. In sharp contrast,
we demonstrate that for groups of odd order,
this barrier vanishes,
causing the critical threshold to collapse to
significantly lower densities bounded by index-$5$
obstructions or the smallest prime divisor.
These results unify and vastly generalize
previous work on cyclic groups,
providing a definitive structural theory
for the transition from sparsity to saturation.
As a decisive application, we resolve a conjecture of
Han and Ren in algebraic coding theory.
By translating the additive rigidity at
density $1/2$ into a geometric constraint,
we prove that for all sufficiently large $q$,
any subset of rational points on an elliptic curve
$E/\mathbb{F}_q$ generating an MDS code must
satisfy the tight bound $|P|\le|E(\mathbb{F}_q)|/2$.

\medskip
\textbf{2020 MSC:} 11B75, 11P70, 20K01, 14H52.

\textbf{Keywords:} Restricted sumsets; critical number;
finite abelian groups; elliptic curves;  MDS codes.
\end{abstract}

\author{}
\date{}
\maketitle

\section{Introduction}
Additive combinatorics studies the interplay between
algebraic structure and the combinatorial behaviour of sumsets.
From the early nineteenth-century inequality
of Cauchy and Davenport in $\mathbb{Z}/p\mathbb{Z}$ to
Kneser's theorem in general abelian groups,
a recurring theme is that  density forces expansion; see, e.g., \cite{Babook,Gry13,Kneser1953,Lev00,Na96,TaoVu2006}.
A particularly subtle and influential variant concerns
restricted sumsets, where summands must be distinct.
This restriction dates back  at least to the celebrated
conjecture of Erd\H{o}s and Heilbronn \cite{EG80},
which sought sharp lower bounds on the size of the
set of sums $a_1+a_2$ with distinct summands
taken from a fixed subset of $\mathbb{Z}/p\mathbb{Z}$.
Dias da Silva and Hamidoune resolved the conjecture
in 1994 using  exterior algebra \cite{DH94},
and shortly thereafter Alon, Nathanson and Ruzsa gave a
polynomial-method proof via the
Combinatorial Nullstellensatz \cite{ANR96}.
Recently, Du and Pan \cite{Du} generalized the
Dias da Silva--Hamidoune theorem to arbitrary finite abelian groups.
These breakthroughs revealed that restricted
sumsets still expand robustly,
but exhibit new obstructions and new extremal
configurations absent
in the unrestricted setting.

While establishing sharp lower bounds on cardinality
is a cornerstone in the theory,
a deeper question concerns the arithmetic completeness of these sumsets.
To formalize this progression,
let $G$ be a finite abelian group and let
$A\subseteq G$.
For an integer $k$ with $0\le k\le |A|$, we write
\[
\Gamma_k(A)=\Bigl\{\alpha_1+\cdots+\alpha_k:
\ \alpha_1,\dots,\alpha_k\in A\ \text{pairwise distinct}\Bigr\}.
\]
 As the length $k$ or the size of $A$ increases,
 the sumset $\Gamma_k(A)$
 transitions from sparse to structured,
 ultimately achieving full coverage of the group.
In their influential classification program for subset sums in
$  \mathbb{Z}/p\mathbb{Z}  $ \cite{NV09}, Nguyen and Vu
formalized this phenomenon into a fundamental classification problem,
 deemed   pivotal
 in additive combinatorics due to its ties to
 both expansion and structural obstructions:
$$
\hbox{For a given $k$, when is $0\in\Gamma_k(A)$
and when is $G=\Gamma_k(A)$?}
$$
This question delineates the additive spectrum,
from local zero-sum solvability to global
representation universality. While the former
drives zero-sum theory,
the latter probes the covering property.
In this paper, we concentrate on the covering aspect:
when does $\Gamma_k(A)=G$,
and how large must $A$ be to guarantee it?
To quantify these thresholds, we employ the (restricted)
\emph{$k$-critical number}
\[
\mu_k(G)
=\min\Bigl\{m\in\mathbb N:\ \text{every }A\subseteq
G\text{ with }|A|\ge m\text{ satisfies }\Gamma_k(A)=G\Bigr\}.
\]

The study of critical numbers has a rich history in combinatorial number theory, as detailed in the systematic treatment by Bajnok \cite{Ba15}. It is well-established that the subgroup structure imposes fundamental barriers on expansion. In groups of even order, the index-$2$ subgroup creates an intrinsic obstruction at density $1/2$. This rigidity was initially characterized by Gallardo, Grekos, et al.  \cite{Ga}, who established that in cyclic groups of even order, the restricted sumset $\Gamma_3(A)$ covers the entire group whenever $|A| > n/2$. Building on their method, Bajnok \cite{Ba15} generalized this result, determining the exact critical number for restricted sumsets of arbitrary length $k$ in even-order cyclic groups. However, prior investigations have highlighted a sharp divergence in the odd-order setting. In the absence of index-$2$ subgroups, the obstruction at density $1/2$ vanishes. Notably, for $k=3$, Lev \cite{Lev02} proved (confirming a conjecture of Gallardo, Grekos, et al.) that the critical threshold collapses to substantially lower densities, governed by index-$5$ configurations rather than index-$2$.

Our first main theorem establishes a universal stability above the highest possible obstruction. While odd-order groups may admit lower thresholds (as suggested by Lev's work), the density $1/2$ remains the absolute barrier for the class of all finite abelian groups (due to the even-order case). We show that once a subset $A$ crosses this maximal index-$2$ barrier, the restricted sumset $\Gamma_k(A)$ covers the entire group $G$ simultaneously for all admissible lengths $k$, regardless of the specific group structure.

\medskip
\noindent\textbf{Theorem A.}\textit{
Let $G$ be a finite abelian group of order $g$,
and let $p(G)$ denote the smallest prime divisor of $g$.
Let $
G[2]=\big\{\alpha\in G:\ 2\alpha=0\big\}
$
denote the $2$-torsion subgroup of $G$.
Assume that $g$ satisfies one of the following conditions:
If $p(G)=2$, assume $g \ge 624\,
    \lvert G[2] \rvert + 1846$;
   if $p(G)=3$, assume $g \ge 3705$;
    if $p(G)=5$, assume $g \ge 6175$;
   if $p(G)\ge 7$, assume $g \ge 46319$.
\medskip
If $A \subseteq G$ satisfies $\lvert A \rvert > g/2$, then
\[
    \Gamma_k(A) = G \qquad \text{for every integer }
    k \text{ with } 3 \le k \le \lvert A \rvert - 3.
\]}
With the global sufficiency of the $1/2$-density established in Theorem A, a natural follow-up question is to identify the exact location of these thresholds when they drop below $1/2$. Our second main theorem addresses this by determining (or bounding) the critical numbers $\mu_k(G)$, confirming the parity-dependent dichotomy for general $k$. Significantly,
for even order, we generalize Bajnok's cyclic result to arbitrary abelian groups, confirming that the index-$2$ obstruction is unique; for odd order, we quantify the collapse of the critical density predicted by Lev's triple-sum result.

\medskip
\noindent\textbf{Theorem B.}\textit{
Let $G$ be a finite abelian group of order $g$.
Let $k$ be an integer satisfying the range conditions associated with the smallest prime divisor $p(G)$ of $g$ such that $3 \le k \le g/p(G)-2$.
\begin{enumerate}[label=\textup{(\roman*)}, leftmargin=2em]
    \item If $g$ is even and $g\ \ge\ 624\,|G[2]|\;+\;1846$, then the $k$-critical number is exactly
    \[
    \mu_k(G) = \frac{g}{2} + 1.
    \]
    \item If $g$ is odd, define the density constant
    \[
    c(g) = \begin{cases}
    2/5 & \text{if } 5 \mid g, \\
    5/13 & \text{if } 5 \nmid g.
    \end{cases}
    \]
Then the $k$-critical number satisfies the upper bound
\[
\mu_k(G) \le \begin{cases}
  \bigl\lfloor c(g)\,g\bigr\rfloor\ +\ 9 & \text{if } p(G)=3~\text{and}~g\geq3\cdot46319, \\[2pt]
\bigl\lfloor c(g)\,g\bigr\rfloor\ +\ 21 & \text{if } p(G)=5~\text{and}~g\geq5\cdot1235, \\[2pt]
\bigl\lfloor c(g)\,g\bigr\rfloor\ +\ 3 & \text{if } p(G)\geq7~\text{and}~g\geq1235. \\[2pt]
    \end{cases}
    \]
\end{enumerate}}

The precise determination of these thresholds has immediate consequences beyond additive combinatorics. In Section~\ref{sec:mds-elliptic}, we demonstrate a striking application to algebraic coding theory.
Let $E/\mathbb{F}_q$ be an elliptic curve and write $G=E(\mathbb{F}_q)$ for its group of $\mathbb{F}_q$-rational points.
Given a set $P\subseteq G$ of evaluation points and a divisor $D$ of degree $k$, the associated elliptic code $C_{\mathcal L}(D,P)$ has dimension $k$ and is MDS if it attains the Singleton bound; see \cite{Ball,Si09,Sti09}.
A key observation of Han and Ren \cite{HR24} is that the MDS property can be characterized by a restricted sumset condition in $G$:
The code fails to be MDS precisely when the point determined by $D$ lies in $\Gamma_k(P)$ (Lemma~\ref{lem:HR-MDS-criterion} below).
Using this criterion, Han and Ren proved that for $3<k<|P|-3$ one has $|P|\le (|G|+5)/2$ and conjectured that the additive constant can be removed for all sufficiently large $q$ \cite{HR24}.

Theorem A directly resolves this conjecture:
for $q$ large enough that $|E(\mathbb{F}_q)|$ meets the quantitative hypotheses of Theorem~A, any subset $P\subseteq E(\mathbb{F}_q)$ with $|P|>|G|/2$ necessarily satisfies $\Gamma_k(P)=G$ for all $3\le k\le |P|-3$, forcing the code to be non-MDS.
Thus the conjectural bound $|P|\le |E(\mathbb{F}_q)|/2$ holds for all sufficiently large $q$, and the additive-combinatorial rigidity at density $1/2$ becomes a structural constraint on MDS elliptic codes.

\medskip
The classification of critical numbers and the rigidity-collapse mechanism constitute a significant advance in additive combinatorics, shifting the focus from density-based bounds to structural characterizations of the subgroup lattice. Furthermore, the techniques developed here---particularly the handling of subgroup interference in the inverse setting---may have broader implications for related fields. In particular, the refined control over subset sums is naturally applicable to problems in zero-sum theory (e.g., variations of the Davenport constant), and the structural constraints imposed by small-index subgroups suggest potential connections to bounding character sums in group representation theory.

The paper is organized as follows.
In Section~\ref{pre} we recall basic notation, establish a normalization lemma for prime quotients, and collect the structural results we use (Dias da Silva--Hamidoune \cite{DH94}, Lev \cite{Lev02}, and Devos--Goddyn--Mohar \cite{De}).
Section~\ref{subsec:proof_A} proves Theorem~A by treating the cases $p(G)\in\{2,3,5\}$ explicitly and then
 performing an inductive argument for $p(G)\ge 7$.
Section~\ref{subsec:proof_B} proves Theorem~B and develops the rigidity/collapse dichotomy for critical numbers.
Finally, Section~\ref{sec:mds-elliptic} applies Theorem~A to MDS elliptic codes and resolves the conjecture of Han and Ren \cite{HR24} for all sufficiently large $q$.

\section{Preliminaries}
\label{pre}
Let $G$ be a finite abelian group of order $g$.
For a subset $A\subseteq G$ and an integer $k$ such that
 $0\le k\le |A|$, we
define the restricted $k$-sumset as
\[
\Gamma_{k}(A)=\Bigl\{\alpha_1+\cdots+\alpha_k:\ \alpha_1,\dots,
\alpha_k\in A\ \text{pairwise distinct}\Bigr\}\subseteq G,
\]
with the convention that   $\Gamma_0(A)=\{0\}$.
The following elementary lemma exhibits
a symmetry property of restricted sumsets.
\begin{lemma}
\label{lem:comp}
Let $G$ be a finite abelian group and let $A \subseteq G$
be a subset of size $a$. Let
$\overline{A} = \sum_{\alpha \in A} \alpha$
denote the sum of the elements of $A$. Then,
for every integer $k$ satisfying $0 \le k \le a$,
we have the identity
\[
\Gamma_k(A) = \overline{A} - \Gamma_{a-k}(A).
\]
Consequently, $\Gamma_k(A) = G$ if and only if $\Gamma_{a-k}(A) = G$. In particular, if $\Gamma_k(A) = G$ for all $1 \le k \le \lfloor a/2 \rfloor$, then $\Gamma_k(A) = G$ for all $1 \le k \le a-1$.
\end{lemma}
\begin{proof}
If $B\subseteq A$ has $|B|=k$ and
$T=A\setminus B$ has $|T|=a-k$, then
$\sum_{\beta\in B}\beta=\overline{A}-
\sum_{\tau\in T}\tau$.
Taking all $k$-subsets $B$ gives
$\Gamma_k(A)\subseteq \overline{A}-\Gamma_{a-k}(A)$.
Conversely, every $(a-k)$-subset $T$ arises as a complement of a $k$-subset,
so the reverse inclusion holds as well.
\end{proof}

While Lemma \ref{lem:comp} allows us to exploit
symmetry to reduce the range of $k$,
our structural analysis will frequently rely
on projecting the set $A$ onto a quotient
group of prime order. To facilitate the inductive
arguments later, it is convenient to normalize
the position of $A$ via translation so that
its densest fiber resides within the kernel.
\begin{lemma}
\label{lem:normal}
Let $p$ be a prime divisor of $|G|$ and let
$\pi:G\twoheadrightarrow \mathbb{Z}/p\mathbb{Z}$ be
a surjective homomorphism with kernel $H$ of size $h$.
If $A\subseteq G$ satisfies $|A| > |G|/2$,
then there exists a translate $A' = A - g_0$ such that
\begin{equation}\label{eq:norm}
A'_0 \subseteq H, \qquad |A'_0| =
\max_{r} |A'_r|, \qquad \text{and} \quad |A'_0| > h/2,
\end{equation}
where $A'_r = A' \cap \pi^{-1}(r)$.
Moreover, for any $1\leq k\leq |A|$,
$\Gamma_k(A')=G$ if and only if $\Gamma_k(A)=G$.
\end{lemma}
\begin{proof}
Let $a_r = |A \cap \pi^{-1}(r)|$.
Since the fibers partition $A$, we have $\sum_{r} a_r = |A|$.
The hypothesis $|A| > ph/2$ implies, by the pigeonhole principle,
that the maximum fiber size satisfies $\max_r a_r > h/2$.
Let $r_0$ be an index achieving this maximum.
Choose any $g_0 \in \pi^{-1}(r_0)$ and set $A' = A - g_0$.
This translation shifts the maximal fiber to the kernel ($r=0$), satisfying \eqref{eq:norm}.
Finally, since $\Gamma_k(A') = \Gamma_k(A) - k g_0$,
the property of covering $G$ is invariant under translation.
\end{proof}

With the normalization machinery established,
we now turn to the structural tools required to verify
the covering property. While symmetry arguments
(Lemma \ref{lem:comp}) effectively handle large values
of $k$, the behavior for small $k$ requires deeper
insights into the algebraic structure of the set.
Specifically, we rely on Lev's celebrated result
\cite{Lev02} regarding restricted triple sums;
we use the explicit numerical form recorded
by Bajnok \cite[Theorem 21]{Ba15}. To state this result, let
$$
G[2]=\big\{\alpha\in G:\ 2\alpha=0\big\}
$$
denote the $2$-torsion subgroup of $G$.
\begin{lemma}
\label{lem:Lev}
Let $G$ be a finite abelian group of order $g$ satisfying
$
g \ge 312\,\abs{G[2]} + 923.
$
Then for every subset $A\subseteq G$, at least one of the following holds:
\begin{enumerate}[leftmargin=2.0em]
\item[\textup{(i)}] $\abs{A}\le \frac{5}{13}g$;
\item[\textup{(ii)}] $A$ is contained in a coset of an index-$2$
subgroup of $G$;
\item[\textup{(iii)}] $A$ is contained in a union of
two cosets of an index-$5$ subgroup of $G$;
\item[\textup{(iv)}] $\Gamma_{3}(A)=G$.
\end{enumerate}
\end{lemma}

While Lemma \ref{lem:Lev} identifies structural
obstructions in general groups, our analysis of
odd-order groups frequently relies on projections onto
quotients of prime order. To handle these projections,
we employ the Dias da Silva--Hamidoune theorem \cite{DH94}
(solving a conjecture made by Erd\H{o}s and Heilbronn \cite{EG80}),
which provides a sharp lower bound on the cardinality of
restricted sumsets in $\mathbb{Z}/p\mathbb{Z}$.
The result was reestablished, using different methods,
e.g. see  \cite[Theorem 3.3]{ANR96}
and \cite{Na96}.

\begin{lemma}
\label{thm:DSH}
Let $p$ be a prime and let
$A\subseteq \mathbb Z/p\mathbb Z$.
For every integer $k$ with $1\le k\le |A|$,
the restricted $k$-sumset satisfies
\[
|\Gamma_k(A)|\ \ge\ \min\{p,\ k|A|-k^2+1\}.
\]
In particular, if $k|A|-k^2+1\ge p$,
then $\Gamma_k(A)=\mathbb Z/p\mathbb Z$.
\end{lemma}

Specializing the bound in Lemma \ref{thm:DSH}
to the high-density regime yields the following well known
result, which guarantees that the restricted
sumset covers the entire group.
We provide a proof below to keep the paper self-contained.
\begin{corollary}\label{cor:half_dense_prime}
Let $p\ge 13$ be a prime and let $A\subseteq \mathbb Z/p\mathbb Z$ satisfy $|A|>p/2$.
Then
\[
\Gamma_k(A)=\mathbb Z/p\mathbb Z\qquad\text{for every integer }k\text{ with }3\le k\le |A|-3.
\]
\end{corollary}
\begin{proof}
Write $a=|A|$. Since $a>p/2$, we have $a\ge (p+1)/2$.
By Theorem~\ref{thm:DSH} with $k=3$,
\[
|\Gamma_3(A)|\ \ge\ \min\{p,\ 3a-8\}.
\]
Moreover,
\[
3a-8\ \ge\ \frac{3(p+1)}{2}-8\ =\ \frac{3p-13}{2}\ \ge\ p
\qquad (p\ge 13),
\]
so $\Gamma_3(A)=\mathbb Z/p\mathbb Z$.
By Lemma~\ref{lem:comp}, we only need to show the result holds for
$3\le k\le \lfloor a/2\rfloor$.
Now fix $k$ with $3\le k\le \lfloor a/2\rfloor$
and set $f(k)=ka-k^2+1$.
Note that the function $f(x) = ax - x^2 + 1$
is an inverted parabola with vertex at $x=a/2$.
Thus, $f(k)$ is nondecreasing on
the integer interval $\{3,\dots,\lfloor a/2\rfloor\}$.
Explicitly, for $k \le \lfloor a/2\rfloor - 1$, we have
\[
f(k+1)-f(k)=a-(2k+1).
\]
Since $2k \le 2(\lfloor a/2\rfloor - 1) \le a-1$,
the difference is nonnegative.
Hence $f(k)\ge f(3)\ge p$,
and Theorem~\ref{thm:DSH} gives
$\Gamma_k(A)=\mathbb Z/p\mathbb Z$
for all $3\le k\le \lfloor a/2\rfloor$.
\end{proof}

We use the following fixed-length subsum
bound for multisets; it is a standard consequence of the
Devos--Goddyn--Mohar theorem \cite{De,Gry13}.

\begin{lemma}[DGM theorem]
\label{lem:DGM-th}
Let $A=(A_1,\dots,A_u)$ be a sequence of finite subsets of a
finite abelian group $G$.
For $1\le \ell\le u$, define
$$\Pi_{\ell}(A) = \left\{ \gamma_{i_1} + \cdots + \gamma_{i_{\ell}} :
1 \leq i_1 < \cdots < i_{\ell} \leq u, ~\gamma_{i_j} \in A_{i_j} \right\}.$$

Let $H = \mathrm{Stab}(\Pi_{\ell}(A)) = \{\beta \in G :
\Pi_{\ell}(A) + \beta = \Pi_{\ell}(A)\}$.
If $\Pi_{\ell}(A) \neq \emptyset$, then
\[ |\Pi_{\ell}(A)| \geq |H| \left( 1 - \ell + \sum_{Q \in G/H}
\min \{ \ell, |\{i \in [1, u] : A_i \cap Q \neq \emptyset \}| \} \right). \]
\end{lemma}
Let $U$ be a multiset of $G$. For any $\alpha \in G$,
let $v_\alpha(U)$ denote the multiplicity of $\alpha$ in $U$.
For $\ell \ge 0$, let $\Sigma_\ell(U) \subseteq G$
denote the set of all restricted subsums of length $\ell$:
$$\Sigma_\ell(U) = \left\{ \sum_{\alpha \in G} c_\alpha \alpha :
\ 0\leq c_\alpha  \in \mathbb{Z},\ \sum_{\alpha \in G} c_\alpha = \ell,
\ \text{and} \ c_\alpha  \le v_\alpha(U)
\text{ for all } \alpha \in G \right\}.$$
We apply Lemma \ref{lem:DGM-th} with $G=\Z_p$ and with each $A_i$
being a singleton ${\alpha_i}$ coming from the sequence
$U=(\alpha_1,\dots,\alpha_u)$, so that $\Pi_\ell(A)=\Sigma_\ell(U)$.

\begin{corollary}
\label{cor:DGMZp}
Let $p$ be prime and let $U=(\alpha_1,\dots,\alpha_u)$
be a sequence in $\Z_p$.
For each $1\le \ell\le u$,
if $\Sigma_{\ell}(U) \neq \mathbb{Z}_p$,
then
\begin{equation}
\label{eq:DGMprime}
|\Sigma_{\ell}(U)| \geq 1 - \ell +
\sum_{\gamma \in \mathbb{Z}_p}
\min \{ \ell, v_\gamma(U) \}.
\end{equation}
\end{corollary}
\begin{proof}
Apply Lemma~\ref{lem:DGM-th} with $G=\Z_p$ and $A_i=\{\alpha_i\}$. Then $\Pi_\ell(A)=\Sigma_\ell(U)$.
The cosets $Q\in \Z_p/H$ correspond to residues modulo $H$; the count
$|{i:A_i\cap Q\neq\emptyset}|$ is exactly the total multiplicity in that coset.
When $H={0}$, the cosets are singletons and the sum becomes $\sum_{x\in\Z_p}\min\{\ell,v_x\}$, giving \eqref{eq:DGMprime}.
\end{proof}

\section{Proofs of Theorems A and B}
This section is devoted to the proofs of the
main results stated in the Introduction.
The argument is organized into two parts.
In Subsection \ref{subsec:proof_A},
we establish Theorem A by developing a
density-based covering strategy that adapts to
the prime factorization of the group order.
Following this, Subsection \ref{subsec:proof_B}
addresses Theorem B, evaluating  the  behavior
of the critical number $\mu_k(G)$.

\subsection{Proof of Theorem A}
\label{subsec:proof_A}
The proof of Theorem A requires us to handle groups of even and odd orders separately. The even order case is structurally more rigid due to the presence of index-$2$ subgroups. To prepare for this analysis, we first record a basic fact: the intersection of distinct index-$2$ cosets is small. This geometric constraint will later prevent large sets from being  trapped  in multiple cosets simultaneously.

\begin{lemma}
\label{lem:intersect}
Let $G$ be a finite abelian group of order $g$.
If $C_1,C_2$ are distinct cosets of index-$2$
subgroups of $G$, then $\abs{C_1\cap C_2}\le g/4$.
\end{lemma}
\begin{proof}
Write $C_i=\gamma_i+H_i$
where $H_i\le G$ has index $2$.
If $H_1=H_2$, then either
$C_1=C_2$ or $C_1\cap C_2=\varnothing$,
so the claim holds trivially.
Assume now $H_1\neq H_2$. Consider the homomorphism
\[
\psi:G\to G/H_1\times G/H_2\cong \mathbb{Z}/2\mathbb{Z}\times \mathbb{Z}/2\mathbb{Z},
\qquad
\psi(\alpha)=(\alpha+H_1,\ \alpha+H_2).
\]
Because $H_1\neq H_2$, choose $\alpha_2\in H_2\setminus H_1$.
Then $\psi(\alpha_2)=(1,0)$ in the product.
Similarly choose $\alpha_1\in H_1\setminus H_2$,
yielding $\psi(\alpha_1)=(0,1)$. Hence $\mathrm{im}(\psi)$
contains $(1,0)$ and $(0,1)$, and therefore $\mathrm{im}(\psi)=\mathbb{Z}/2\mathbb{Z}\times \mathbb{Z}/2\mathbb{Z}$.
Consequently,  $\ker(\psi)=H_1\cap H_2$ has index $4$,
so $|H_1\cap H_2|=g/4$.
Now $C_1\cap C_2$ is either empty or a coset of $H_1\cap H_2$
(if $\alpha\in C_1\cap C_2$ then $C_1=\alpha+H_1$ and $C_2=\alpha+H_2$,
and thus
$C_1\cap C_2=\alpha+H_1\cap H_2$);
in the latter case
$|C_1\cap C_2|=|H_1\cap H_2|=g/4$.
In all cases, $|C_1\cap C_2|\le g/4$.
\end{proof}

Before distinguishing between the cases of even and odd order,
we first establish a general density result.
The following lemma demonstrates that if a set $A$
is strictly larger than half the group,
the restricted sumsets $\Gamma_k(A)$ automatically cover
$G$ for a wide range of $k$ near the boundaries.

\begin{lemma}
\label{lem:3-4-5}
Let $G$ be a finite abelian group of order
$
g\ge 312\,|G[2]|+923.
$
Let $A\subseteq G$ with $|A|=a>g/2$, and define
\[
d=a-\Bigl\lceil \frac{2g}{5}\Bigr\rceil+2\quad(\text{note that}~
d>100).
\]
Then
\[
\Gamma_k(A)=G \qquad\text{for every integer } k \text{ with } 3\le k\le d
~\text{and}~a-d\le k\le a-3.
\]
\end{lemma}
\begin{proof}
By Lemma \ref{lem:comp},
we only need to show that the lemma holds in the range $3\le k\le d$.
Fix an integer $k$ with $3\le k\le d$, and put $t=k-3\ge 0$.
Then
\begin{equation}\label{eq:Bsize_lower_updated}
a-t=a-(k-3)\ge \Bigl\lceil \frac{2g}{5}\Bigr\rceil+1>\frac{2g}{5}>\frac{g}{4}.
\end{equation}
\textbf{Claim.}
There exist distinct $\alpha_1,\dots,\alpha_t\in A$ such that
$B=A\setminus\{\alpha_1,\dots,\alpha_t\}$ is not contained
in any coset of an index-$2$
subgroup of $G$.

\smallskip\noindent
\emph{Proof of the claim.}
If $t=0$, then $B=A$ and $A$ cannot lie in any
index-$2$ coset because
the size $|A|>g/2$. So assume $t\ge 1$.
Suppose for contradiction that for every
$t$-subset $T\subseteq A$ the set
$A\setminus T$ is contained in a coset of an
index-$2$ subgroup of $G$.
By \eqref{eq:Bsize_lower_updated},
for every such $T$ we have $|A\setminus T|>g/4$.
In particular, the index-$2$ coset
containing $A\setminus T$ is unique:
if $A\setminus T\subseteq C_1\cap C_2$
for index-$2$ cosets $C_1,C_2$, then
\[
|C_1\cap C_2|\ge |A\setminus T|>\frac{g}{4},
\]
so Lemma~\ref{lem:intersect} forces $C_1=C_2$.
Thus we may denote by $C_T$ the unique index-$2$ coset with
\[
A\setminus T\subseteq C_T.
\]
Now let $T_1,T_2\subseteq A$ be $t$-subsets
with $|T_1\cap T_2|=t-1$
(i.e.\ $T_2$ is obtained from $T_1$
by replacing one element).
Such a pair exists because $t\ge 1$
and $t<a$.
Then $|T_1\cup T_2|=t+1$ and
\[
(A\setminus T_1)\cap
(A\setminus T_2)=A\setminus (T_1\cup T_2),
~~\text{so}~
|A\setminus (T_1\cup T_2)|=a-(t+1).
\]
Using \eqref{eq:Bsize_lower_updated} we obtain
\[
a-(t+1)=(a-t)-1\ge
\Bigl\lceil \frac{2g}{5}\Bigr\rceil>\frac{g}{4}.
\]
Since
\[
A\setminus (T_1\cup T_2)\subseteq
(A\setminus T_1)\cap (A\setminus T_2)
\subseteq C_{T_1}\cap C_{T_2},
\]
we have $|C_{T_1}\cap C_{T_2}|>g/4$,
and Lemma~\ref{lem:intersect} implies
$
C_{T_1}=C_{T_2}.
$
Next, we show that $C_T$ is the same for all
$t$-subsets $T\subseteq A$.
Let $T,T'\subseteq A$ be arbitrary $t$-subsets.
Write
\[
T\setminus T'=\{\tau_1,\dots,\tau_r\},
\qquad T'\setminus T=\{\tau_1',\dots,\tau_r'\}
\]
for some $r\ge 0$. Define a
sequence of $t$-subsets by
\[
T_0=T,\qquad T_i=(T_{i-1}\setminus\{\tau_i\})\cup
\{\tau_i'\}\quad (1\le i\le r).
\]
Then $T_r=T'$, and for each $i$ we have
$|T_{i-1}\cap T_i|=t-1$, so the previous
paragraph gives $C_{T_{i-1}}=C_{T_i}$.
Hence $C_T=C_{T'}$.
Thus there exists a fixed index-$2$ coset
$C$ such that $A\setminus T\subseteq C$
for every $t$-subset $T\subseteq A$.

Finally, fix any $\gamma\in A$. Since $t<a$,
we can choose a $t$-subset $T\subseteq A$
with $\gamma\notin T$.
Then $\gamma\in A\setminus T\subseteq C$,
so $A\subseteq C$.
But $|C|=g/2$ while $|A|>g/2$,
a contradiction. This proves the claim.

\medskip\noindent
Choose $\alpha_1,\dots,\alpha_t$ as in the claim,
and set $B=A\setminus\{\alpha_1,\dots,\alpha_t\}$.
Then $|B|=a-t$, so by
\eqref{eq:Bsize_lower_updated} we have $|B|>2g/5$.
In particular $|B|>5g/13$,
so alternative (i) of Lemma~\ref{lem:Lev} fails.
By construction, $B$ is not
contained in any index-$2$ coset, so alternative (ii) fails.
Also, $B$ cannot be contained in a
union of two cosets of an index-$5$ subgroup,
since every such union has size at most $2g/5$;
hence alternative (iii) fails.
Therefore Lemma~\ref{lem:Lev} applied to $B$ yields
$
\Gamma_3(B)=G.
$
Now let $\gamma\in G$ be arbitrary.
Since $\Gamma_3(B)=G$, there exist distinct
$\beta_1,\beta_2,\beta_3\in B$ such that
\[
\beta_1+\beta_2+\beta_3=\gamma-(\alpha_1+\cdots+\alpha_t).
\]
Then $\alpha_1,\dots,\alpha_t,\beta_1,\beta_2,\beta_3$
are $t+3=k$ distinct elements of $A$ whose sum is $\gamma$.
Thus $\gamma\in \Gamma_k(A)$, and since $\gamma$ is
arbitrary we conclude $\Gamma_k(A)=G$.
This completes the proof for all $3\le k\le d$.
\end{proof}

Lemma \ref{lem:3-4-5} establishes the result for $k$
near the boundaries. To bridge the remaining gap,
we distinguish cases based on the smallest prime divisor $p(G)$ of $|G|$.
We begin with the rigid case $p(G)=2$,
where the threshold is controlled by a subgroup of index 2.

\begin{proposition}
\label{p=2}
Let $G$ be a finite abelian group and $H \le G$
a subgroup of index 2 with $|H|=h$.
Let $A \subseteq G$ satisfy $|A| > |G|/2$.
Assume
$h \ge 312|H[2]| + 923$.
Then for every integer $k$ with $6 \le k \le |A| - 6$,
we have $\Gamma_k(A) = G$.
\end{proposition}
\begin{proof}
Decompose $A$ into $A_0 = A \cap H$ and $A_1 = A \setminus H$, denoting their sizes by $a_0$ and $a_1$.
By Lemma \ref{lem:normal}, we may assume that the kernel contains the majority of elements, i.e., $a_0 > h/2$.
In light of Lemma \ref{lem:comp}, it suffices to verify $\Gamma_k(A) = G$ for the range $6 \le k \le \lfloor a/2 \rfloor$.
Fix an arbitrary target $\alpha \in G$ and let $s = \pi(\alpha) \in \mathbb{Z}/2\mathbb{Z}$.
We distinguish two cases based on the cardinality
of the external fiber $A_1$.

\paragraph{Case 1:  $a_1 \le k-3$.}
In this case,
we choose a subset $L \subseteq A_1$ of size $\ell$
such that $\ell \equiv s \pmod 2$ and
$\ell \in \{a_1, a_1-1\}$. Since $|A| > h$
implies $A \not\subseteq H$, we have $a_1 \ge 1$,
so such a choice is always possible.
Let $\overline{L}= \sum_{\iota \in L} \iota$
($\overline{L}=0$ if $\ell=0$).
Note that $\pi(\overline{L})\equiv s \pmod 2$,
so $\alpha-\overline{L} \in H$.
We define
$t = k - 3 - \ell$.
Clearly,  $t \ge 0$ and $a_0 - t \ge a_0 + a_1 - k + 2>0$.
We choose any set $T \subseteq A_0$ of size $t$.
Let $B = A_0 \setminus T$.
We must ensure $B$ is dense enough to apply Lemma \ref{lem:3-4-5}.
Since $|B| = a_0 - t \geq a - k + 2$
and $k \le a/2$,
it follows that $a - k+2 \ge a/2 +2> h/2$. Thus $|B| > h/2$.
By Lemma \ref{lem:3-4-5} applied to $B \subseteq H$,
we have $\Gamma_3(B) = H$.
Since $\alpha -  \overline{L} - \sum_{\tau \in T} \tau \in H$,
there exist three distinct elements
$\beta_1, \beta_2, \beta_3 \in B$ summing to this value.
The union $L \cup T \cup \{\beta_1, \beta_2, \beta_3\}$
comprises $\ell + t + 3 = k$ distinct elements summing to $\alpha$.

\paragraph{Case 2: $a_1 > k-3$.}
In this case, $A_1$ is large relative to $k$.
We need to select $\ell$ elements from $A_1$
and $m$ elements from $A_0$ such that $\ell + m = k$
and $\ell \equiv s \pmod 2$. We choose
$m \in \{3, 4\}$ depending on the parity of $k-3$:
\begin{itemize}
    \item If $k-3 \equiv s \pmod 2$, set $m=3$ and $\ell=k-3$.
    \item If $k-3 \not\equiv s \pmod 2$,
    set $m=4$ and $\ell=k-4$ (note $k-4 \equiv s \pmod 2$).
\end{itemize}
Since $k \ge 6$, both choices of $\ell$ are positive.
Furthermore, since $a_1 > k-3$,
we have $a_1 > \ell$ in both cases,
so distinct elements can be chosen.
Select any subset $L \subseteq A_1$ of size $\ell$.
Let $\overline{L}= \sum_{\iota \in L} \iota$.
Then $\alpha- \overline{L} \in H$.
 By Lemma \ref{lem:3-4-5},
 $\Gamma_m(A_0) = H$ for $m \in \{3, 4\}$.
Thus, there exist distinct elements
$\gamma_1, \dots, \gamma_m \in A_0$ such that
$\sum \gamma_i = \alpha - \overline{L}$.
The set $L \cup \{\gamma_1, \dots, \gamma_m\}$ contains $\ell+m=k$
distinct elements summing to $\alpha$.
\end{proof}

Having disposed of the case where $G$ has even order,
we now turn to the case where  $3$ is a divisor of
$|G|$. The strategy remains to analyze the
projection of $A$ onto a quotient of prime order.
However, unlike the index-2 case where parity constraints
played a central role, the arithmetic in
$\mathbb{Z}/3\mathbb{Z}$ requires a different
covering property.
The following  lemma asserts that the union of
sumsets of three consecutive lengths is sufficient
to cover the cyclic group of order 3.

Before proceeding, we recall the notation for
restricted sumsets of a multiset.
Let $U$ be a multiset of elements from an abelian group $G$.
For any $\alpha \in G$, let $v_\alpha(U)$ denote
the multiplicity of $\alpha$ in $U$.
For an integer $\ell \ge 0$,
let $\Sigma_\ell(U)$ denote the set of all sums of $\ell$
elements of $U$:
\[
\Sigma_\ell(U) = \left\{ \sum_{\alpha \in G} c_\alpha \alpha :
\ c_\alpha  \in \mathbb{Z}, \ 0\leq c_\alpha \le v_\alpha(U)
\text{ for all } \alpha \in G,
\ \text{and} \ \sum_{\alpha \in G} c_\alpha = \ell \right\}.
\]
\begin{lemma}
\label{lem:Z3}
Let $U$ be a multiset of elements from
$\mathbb{Z}/3\mathbb{Z}$ consisting of $a_1$ copies
of $1$ and $a_2$ copies of $2$. Let $u= a_1 + a_2\geq 2$.
Then for any integer $\ell$ such that $0 \le \ell \le u-2$,
\[
\Sigma_\ell(U) \cup \Sigma_{\ell+1}(U) \cup \Sigma_{\ell+2}(U) = \mathbb{Z}/3\mathbb{Z}.
\]
\end{lemma}
\begin{proof}
If $a_1=0$, $U$ consists only of $2$'s.
The sum of $k$  elements in $U$ is $2k \pmod 3$.
The union involves sums $2\ell, 2(\ell+1), 2(\ell+2)$.
Since $2$ is coprime to $3$, these are three
distinct residues modulo $3$, covering the group.
The case $a_2=0$ is identical
(sums are $\ell, \ell+1, \ell+2 \pmod 3$).

Assume $a_1 \ge 1$ and $a_2 \ge 1$.
For any length $k$ ($0 \le k \le u$),
a sum involves $x$ copies of $1$ and $y$ copies of
$2$ with $x+y=k$. The sum is
$x+2y = (k-y)+2y \equiv k+y \pmod 3$.
Thus, $\Sigma_k(U) = \{ k+y \pmod 3 \mid y \in I(k) \}$,
where the feasible interval for $y$ (number of $2$'s chosen) is
\[
I(k)= [\max(0, k-a_1), \min(k, a_2)] \cap \mathbb{Z}.
\]
The size of this interval is $|I(k)| = \min(k, a_2) - \max(0, k-a_1) + 1$.
If $|I(k)| \ge 3$, then $\Sigma_k(U) = \mathbb{Z}/3\mathbb{Z}$.
We distinguish two cases based on the abundance of the minority element:

\paragraph{Case 1: $\min(a_1, a_2) \ge 2$.}
We first show that
if $2 \le k \le u-2$, then $|I(k)| \ge 3$.
Indeed, the length is $\min(k, a_2) - \max(0, k-a_1) + 1$.
    Since $k \le u-2 = a_1+a_2-2$, we have $k-a_1 \le a_2-2$.
  If $k \le a_1$ and $k \le a_2$, the length is $k+1 \ge 3$.
      If $k > a_1$ and $k > a_2$, the length is
      $a_2 - (k-a_1) + 1 = u - k + 1 \ge 3$.
Intermediate cases similarly yield
length $\ge \min(a_1, a_2)+1 \ge 3$.

We claim that for any $\ell$ in the valid
range $0 \le \ell \le u-2$, at least one index
$k \in \{\ell, \ell+1, \ell+2\}$ satisfies $|I(k)| \ge 3$.
Consider the triplet of indices $\{\ell, \ell+1, \ell+2\}$.
\begin{itemize}
    \item If $\ell=0$, the indices are $\{0, 1, 2\}$.
    $k=2$ satisfies $2 \le k \le u-2$ (as $u \ge 4$).
    Thus $\Sigma_2(U) = \mathbb{Z}/3\mathbb{Z}$.
    \item If $\ell > 0$, then $k=\ell+1$
    satisfies $2 \le \ell+1$.
    We check the upper bound:
    If $\ell+1 \le u-2$, then $\Sigma_{\ell+1}(U)$
    covers the group.
    If $\ell+1 > u-2$, since $\ell \le u-2$,
    this implies $\ell = u-2$.
    The indices are $\{u-2, u-1, u\}$.
    Here $k=u-2$ satisfies the condition.
    Thus $\Sigma_{u-2}(U)$ covers the group.
\end{itemize}
In all subcases, the union covers $\mathbb{Z}/3\mathbb{Z}$.

\paragraph{Case 2: $\min(a_1, a_2) = 1$.}
Without loss of generality, let $a_2=1$
and $a_1 \ge 1$. Then $u = a_1+1$.
For any $k$ such that $1 \le k \le a_1 = u-1$:
We can choose zero `2's
    (since $k \le a_1$). Sum: $k \cdot 1 = k$.
We can choose one `2' (since $k \ge 1$).
    Sum: $(k-1) \cdot 1 + 2 = k+1$.
Thus, for $1 \le k \le u-1$,
$\Sigma_k(U) \supseteq \{k, k+1 \pmod 3\}$.
Now consider the union over indices $\ell, \ell+1, \ell+2$:
\begin{itemize}
    \item If $\ell=0$: Indices are $\{0, 1, 2\}$.
    We have $\Sigma_0 = \{0\}$ and
    $\Sigma_1 \supseteq \{1, 2\}$, so the
    union is $\{0, 1, 2\}$.
    \item If $\ell \ge 1$: Indices are $\{\ell, \ell+1, \ell+2\}$.
    We have $\Sigma_\ell \supseteq \{\ell, \ell+1\}$ and
    $\Sigma_{\ell+1} \supseteq \{\ell+1, \ell+2\}$.
    The union contains $\{\ell, \ell+1, \ell+2\}$, which are three consecutive integers covering $\mathbb{Z}/3\mathbb{Z}$.
\end{itemize}
\end{proof}

Armed with this combinatorial tool for $\mathbb{Z}/3\mathbb{Z}$,
we proceed to the structural proof for the case where  $3$
is a   divisor of $|G|$.
\begin{proposition}
\label{prop:p3}
Let $G$ be a finite abelian group and $H \le G$ a
subgroup of index $3$
with $|H|=h$. Assume  $h \ge 312|H[2]| + 923$.
Let $A \subseteq G$ satisfy $|A| > |G|/2= 3h/2$.
Then for every integer $k$ with $6 \le k \le |A| - 6$, we have $\Gamma_k(A) = G$.
\end{proposition}

\begin{proof}
Let $\pi: G \to \mathbb{Z}/3\mathbb{Z}$ be the quotient map.
Decompose $A$ into fibers
$A_r = A \cap \pi^{-1}(r)$ for $r \in \{0, 1, 2\}$.
Let $a_r = |A_r|$ and $a = |A|$.
By Lemma \ref{lem:normal}, we can assume that
$a_0 > h/2$.
Define the external set $A_{out} = A_1 \cup A_2$ with
size $u=a_1+a_2$.
Let $U$ be the multiset of residues of $A_{out}$.
By Lemma \ref{lem:comp}, it suffices to prove $\Gamma_k(A)=G$ for
$
6 \le k \le \left\lfloor \frac{a}{2} \right\rfloor.
$
Fix a target $\alpha \in G$ with residue $s = \pi(\alpha)$.
We distinguish two cases based on the
availability of external elements:

\paragraph{Case 1: $u\ge k-3$.}
In this case, $1\leq k-5\leq u-2$
($u>2$ because $u=a-a_0\geq 3h/2-h=h/2$).
By Lemma \ref{lem:Z3}, there exists a length
$\ell \in  \{k-5, k-4, k-3\}$
such that we can choose a subset
$L \subseteq A_{out}$ with $|L|=\ell$ and $\pi(\sum L) = s$.
Let $\overline{L} = \sum_{\iota \in L} \iota$.
Then $\alpha - \overline{L} \in H$.
Define  $m = k - \ell$.
Note that $m \in \{3, 4, 5\}$.
We need to find $m$ distinct elements in
$A_0$ summing to $\alpha - \overline{L}$.
It follows from $a_0 > h/2$ and Lemma \ref{lem:3-4-5}
that
there exists $M \subseteq A_0$ with $|M|=m$
summing to the required kernel element.
The set $L \cup M$ is a valid $k$-subset of $A$.

\paragraph{Case 2: $u < k-3$.}
By Lemma \ref{lem:Z3}, there exists
$\ell \in \{u-2, u-1, u\}$ and a subset $L \subseteq A_{out}$
of size $\ell$ such that $\pi(\sum L) = s$.
Let $\overline{L} = \sum_{\iota \in L} \iota$.
We need to choose $k - \ell$ elements from $A_0$.
Let $t = k - 3 - \ell$.
Since $\ell \le u < k-3$, we have $t > 0$.
Choose any arbitrary subset $T \subseteq A_0$ of size $t$
(clearly, $t<a_0$, see the calculations below).
Let $\overline{T} = \sum_{\tau \in T} \tau$ and
$B = A_0 \setminus T$. We  verify that
$B$ is dense enough to generate $H$:
\[
|B| = a_0 - t = a_0 - (k - 3 - \ell) = a_0 + \ell - k + 3.
\]
We use the lower bound $\ell \ge u - 2$:
\[
|B| \ge a_0 + (u - 2) - k + 3 = (a_0 + u) - k + 1 = a - k + 1.
\]
Recall the range condition $k \le \lfloor a/2 \rfloor$, so
$
|B| \ge a - \frac{a}{2} + 1 = \frac{a}{2} + 1.
$
Using the hypothesis $|A| = a > 3h/2$:
\[
|B| > \frac{3h/2}{2} + 1 = \frac{3h}{4} + 1.
\]
We have $|B| > h/2$.
Therefore, by Lemma \ref{lem:3-4-5}, $\Gamma_3(B) = H$.
There exist distinct $\beta_1, \beta_2, \beta_3 \in B$
such that $\beta_1+\beta_2+\beta_3 = \alpha - \overline{L} - \overline{T}$.
The union $L \cup T \cup \{\beta_1, \beta_2, \beta_3\}$ provides the required solution.
\end{proof}


We have now established the results for groups
where the smallest prime divisor is $2$ or $3$.
The final step in our inductive strategy is to handle
the case where the smallest prime divisor is $p=5$.
As before, the proof relies on the additive structure
of the quotient group. The following lemma asserts that
any four non-zero elements in $\mathbb{Z}/5\mathbb{Z}$
suffice to cover the entire group via subset sums.
While this is a standard fact--and a direct consequence of the
Cauchy--Davenport theorem
(see, e.g.,
\cite[Theorem 6.2]{Gry13})--we include a proof for completeness.

\begin{lemma}
\label{lem:four-nonzero}
Let $r_1, r_2, r_3, r_4 \in
\mathbb{Z}/5\mathbb{Z} \setminus \{0\}$
be non-zero elements (not necessarily distinct).
Let $R$ be the multiset consisting of $r_1,r_2,r_3$ and $r_4$.
Then the set of all subsums of $R$ (of lengths $0,1,2,3,4$) is the whole group:
\[
\Sigma_0(R)\cup \Sigma_1(R)\cup \Sigma_2(R)\cup \Sigma_3(R)\cup \Sigma_4(R)
=\mathbb{Z}/5\mathbb{Z}.
\]
\end{lemma}
\begin{proof}
For $j=1,2,3,4$ define multisets
\[
S_0=\{0\},\qquad S_j=S_{j-1}+\{0,r_j\}=S_{j-1}\cup
(S_{j-1}+r_j)\subseteq \mathbb{Z}/5\mathbb{Z}.
\]
Then $S_4$ is exactly the multiset of all subset
sums of $r_1,r_2,r_3$ and $r_4$, and
\[
\Sigma_0(R)\cup \Sigma_1(R)\cup \Sigma_2(R)\cup \Sigma_3(R)\cup \Sigma_4(R)
=S_4.
\]
We claim that if
$\emptyset\subsetneq S\subsetneq \mathbb{Z}/5\mathbb{Z}$
and $r\neq 0$, then
$S+r\neq S$. Indeed, if $S+r=S$,
then by iterating we get
$S+mr=S$ for all integers $m$,
but $r$ generates $\mathbb{Z}/5\mathbb{Z}$,
so $S$ is invariant under all translations and
hence must equal the whole group,
contradicting $S\subsetneq \mathbb{Z}/5\mathbb{Z}$.
Therefore, whenever $S_{j-1}\neq \mathbb{Z}/5\mathbb{Z}$
we have
$S_{j-1}+r_j\neq S_{j-1}$ and thus
$|S_j|=|S_{j-1}\cup (S_{j-1}+r_j)|\ge |S_{j-1}|+1$
(here $|X|$ denotes the number of distinct elements in the multiset $X$).
Starting from $|S_0|=1$, we obtain $|S_4|\ge 5$.
Since $\mathbb{Z}/5\mathbb{Z}$ has exactly $5$ elements,
this forces
$S_4=\mathbb{Z}/5\mathbb{Z}$, proving the claim.
\end{proof}

We now apply Lemma \ref{lem:four-nonzero} to
derive a covering property for sumsets of consecutive
lengths in multisets of arbitrary size $u \ge 4$.
\begin{lemma}
\label{lem:Z5}
Let $U$ be a multiset of elements of
$\mathbb{Z}/5\mathbb{Z}\setminus\{0\}$
with total size $u\ge 4$. Then for every integer
$\ell$ with $0\le \ell\le u-4$,
\[
\Sigma_\ell(U)\cup \Sigma_{\ell+1}(U)\cup
\Sigma_{\ell+2}(U)\cup
\Sigma_{\ell+3}(U)\cup \Sigma_{\ell+4}(U)
=\mathbb{Z}/5\mathbb{Z}.
\]
\end{lemma}
\begin{proof}
Fix $\ell$ with $0\le \ell\le u-4$.
Choose any submultiset $L\subseteq U$
of size $\ell$, and write
$\overline{L}=\sum_{\iota\in L}\iota\in \mathbb{Z}/5\mathbb{Z}$.
Remove $L$ from $U$, leaving a multiset $U'$
of size $u-\ell\ge 4$.
Pick any four elements $r_1,r_2,r_3,r_4$ from $U'$.
By Lemma \ref{lem:four-nonzero},
the subsums of $\{r_1,r_2,r_3,r_4\}$
of lengths $0,1,2,3,4$
cover $\mathbb{Z}/5\mathbb{Z}$. Hence for any
$s\in \mathbb{Z}/5\mathbb{Z}$ there exists an
integer $t\in\{0,1,2,3,4\}$
and a submultiset
$T\subseteq\{r_1,r_2,r_3,r_4\}$ of size $t$
such that $\overline{T}=\sum_{\tau\in T}\tau=s-\overline{L}$.
Then $L\cup T$ is a submultiset of $U$ of size $\ell+t$ and sum
$\overline{L}+\overline{T}=s.
$
Therefore $s\in \Sigma_{\ell+t}(U)$ for some $t\in\{0,1,2,3,4\}$, proving
that the stated union equals $\mathbb{Z}/5\mathbb{Z}$.
\end{proof}

Equipped with this covering property for
$\mathbb{Z}/5\mathbb{Z}$, we proceed to the main result
for the case where $G$ has a subgroup of index $5$.
\begin{proposition}\label{prop:p5}
Let $G$ be a finite abelian group and $H\le G$ a subgroup of index $5$
with $|H|=h$. Assume $h\ge 312|H[2]|+923$.
Let $A\subseteq G$ satisfy $|A|>|G|/2=5h/2$.
Then for every integer $k$ with $6\le k\le |A|-6$, we have $\Gamma_k(A)=G$.
\end{proposition}

\begin{proof}
Let $a=|A|$.
By Lemma \ref{lem:comp},
it suffices to prove $\Gamma_k(A)=G$
for all integers $k$ with
$
6\le k\le \left\lfloor \frac{a}{2}\right\rfloor.
$
Fix such a $k$ and fix an arbitrary target element $\alpha\in G$.
Let $\pi:G\twoheadrightarrow G/H\cong \mathbb{Z}/5\mathbb{Z}$
be the quotient map.
For each $r\in \mathbb{Z}/5\mathbb{Z}$ define the fibers
\[
A_r=A\cap \pi^{-1}(r),\qquad a_r=|A_r|.
\]
According to Lemma \ref{lem:normal},
we can assume that
$a_0>h/2$.
Set
\[
A_{\mathrm{out}}=A_1\cup A_2\cup A_3\cup A_4,
\qquad u=|A_{\mathrm{out}}|=a-a_0.
\]
Note that $a_0\le h$, hence
\[
u=a-a_0 \ge a-h > \frac{5h}{2}-h=\frac{3h}{2}.
\]
As $a>5h/2$,  we have $a/2>5h/4$ and thus
 $a-h > a/2 \ge k$ and thus
$
u>k.
$
We will use Lemma \ref{lem:3-4-5} inside the subgroup $H$:
since $h\ge 312|H[2]|+923$ and $|A_0|=a_0>h/2$,
Lemma \ref{lem:3-4-5} implies
\begin{equation}\label{eq:internal}
\Gamma_m(A_0)=H \qquad\text{for every } m\in\{3,4,5,6,7\}.
\end{equation}

Let $s=\pi(\alpha)\in \mathbb{Z}/5\mathbb{Z}$.

\medskip\noindent
\textbf{Case 1: $k=6$.}
We first show that $A_{\mathrm{out}}$ meets at
least two nonzero residue classes.
Indeed, if $A_{\mathrm{out}}\subseteq \pi^{-1}(r)$
for a single $r\neq 0$, then
$A\subseteq \pi^{-1}(0)\cup \pi^{-1}(r)$, so $|A|\le 2h$,
contradicting $|A|>5h/2$.
By the pigeonhole principle, among the four
fibers $A_1,A_2,A_3,A_4$ there exists
$r\in\{1,2,3,4\}$ with
\[
a_r\ge \left\lceil \frac{u}{4}\right\rceil > 3.
\]
Choose distinct elements $x_1,x_2,x_3\in A_r$.
Choose $r'\in\{1,2,3,4\}\setminus\{r\}$ with $a_{r'}\ge 1$
and pick $z\in A_{r'}$.
Consider the set of residues obtainable as sums
of at most three of these four elements.
In $\mathbb{Z}/5\mathbb{Z}$ this set contains
\[
0,\quad r,\quad 2r,\quad 3r,\quad r',\quad r'+r,\quad r'+2r.
\]
Since $r\neq 0$, multiplication by $r^{-1}$
is a permutation of $\mathbb{Z}/5\mathbb{Z}$.
Let $\tau=r^{-1}r'\in\{2,3,4\}$.
Then the displayed set becomes
\[
\{0,1,2,3,\tau,\tau+1,\tau+2\}.
\]
If $\tau=2$ this is $\{0,1,2,3,4\}$;
if $\tau=3$ then $\tau+1=4$ and we still get
$\{0,1,2,3,4\}$; and if $\tau=4$ it is again $\{0,1,2,3,4\}$.
Thus there exists a subset $L\subseteq \{x_1,x_2,x_3,z\}\subseteq A_{\mathrm{out}}$
of some size $\ell\in\{0,1,2,3\}$ such that
$
\pi\bigl(\sum_{\iota\in L} \iota\bigr)=s.
$
Let $\overline{L}=\sum_{\iota\in L} \iota\in G$.
Then $\alpha-\overline{L}\in H$.
Put $m=6-\ell\in\{3,4,5,6\}$. By \eqref{eq:internal},
$\Gamma_m(A_0)=H$, hence
there exist distinct $\gamma_1,\dots,\gamma_m\in A_0$ such that
\[
\gamma_1+\cdots+\gamma_m=\alpha-\overline{L}.
\]
All elements of $L$ lie outside $A_0$,
while all $\gamma_i$ lie in $A_0$, so the $6$ elements
$
L\ \cup\ \{\gamma_1,\dots,\gamma_m\}
$
are pairwise distinct and sum to $\alpha$.
Therefore $\alpha\in \Gamma_6(A)$.

\medskip\noindent
\textbf{Case 2: $k\ge 7$.}
Let $U$ be the multiset of residues of $A_{\mathrm{out}}$ in $\mathbb{Z}/5\mathbb{Z}$;
equivalently, $U$ contains $a_r$ copies of $r$ for each $r\in\{1,2,3,4\}$.
Then the multiset $U$ has size $u$.
Set $\ell_0=k-7\ge 0$. Since $u>k$, we have
$\ell_0=k-7< u-4$, so Lemma \ref{lem:Z5}
applies to $U$ at $\ell=\ell_0$.
Hence there exist an integer
$\ell\in\{\ell_0,\ell_0+1,\ell_0+2,\ell_0+3,\ell_0+4\}
=\{k-7,k-6,k-5,k-4,k-3\}$
and a subset $L\subseteq A_{\mathrm{out}}$ of size $|L|=\ell$
such that
$
\pi\bigl(\sum_{\iota\in L} \iota\bigr)=s.
$
Let $\overline{L}=\sum_{\iota\in L} \iota\in G$.
Then $\alpha-\overline{L}\in H$.
Put $m=k-\ell$, so $m\in\{3,4,5,6,7\}$.
By \eqref{eq:internal}, $\Gamma_m(A_0)=H$,
so there exist distinct $\gamma_1,\dots,\gamma_m\in A_0$ with
\[
\gamma_1+\cdots+\gamma_m=\alpha-\overline{L}.
\]
Again $L\subseteq A_{out}$ and
$\{\gamma_1,\dots,\gamma_m\}\subseteq A_0$, so all $k$ elements
in $L\cup\{\gamma_1,\dots,\gamma_m\}$ are
pairwise distinct and sum to $\alpha$.
Thus $\alpha\in \Gamma_k(A)$.
Since $\alpha\in G$ is arbitrary,
we have shown $\Gamma_k(A)=G$ for all
$k$ with $6\le k\le \lfloor a/2\rfloor$.
\end{proof}

Having settled the specific cases $p(G) \in \{2,3,5\}$,
we now turn to the general inductive step for an
arbitrary prime index $p(G) \ge 7$.
In the general argument,
we will need to estimate the size of sumsets
in the quotient group using the fiber sizes.
The following   lemma provides a sharp lower bound for this purpose,
asserting that the sum of truncated multiplicities
is minimized precisely when the mass is concentrated in
as few fibers as possible.

\begin{lemma}
\label{lem:minimize}
Fix integers $p\ge 2$, $h\ge 1$
and $1\le \ell\le h$.
Let $(v_1,\dots,v_{p-1})$ be integers
with $0\le v_i\le h$ and $\sum_{i=1}^{p-1} v_i=u$.
Write $u=qh+r$ with $q=\lfloor u/h\rfloor$ and $0\le r<h$.
Then
\[
\sum_{i=1}^{p-1}\min(\ell,v_i)\ \ge\ q\ell+\min(\ell,r).
\]
\end{lemma}
\begin{proof}
Let $f(t)=\min(\ell,t)$ on $\{0,1,\dots,h\}$.
Its discrete slope drops from $1$ to $0$ at $t=\ell$,
hence $f$ is concave in the discrete sense:
\[
f(t+1)-f(t)\ \le\ f(t)-f(t-1)\qquad (1\le t\le h-1).
\]
Take any feasible vector $(v_i)$. If there exist
indices $i\neq j$ with $0<v_i<h$ and $0<v_j<h$, then
(after relabeling) we may assume $v_i\le v_j$ and define
\[
v_i'=v_i-1,\quad v_j'=v_j+1,\quad v_k'=v_k\ (k\neq i,j).
\]
This preserves the constraints and, by discrete concavity,
\[
f(v_i-1)+f(v_j+1)\ \le\ f(v_i)+f(v_j),
\]
so $\sum_k f(v_k')\le \sum_k f(v_k)$.
Iterating this operation pushes mass to the largest
coordinates without increasing the sum. The process
terminates at an extreme vector with $q$ coordinates
equal to $h$, one coordinate equal to $r$, and the
rest $0$. For that vector the sum equals
$q f(h)+f(r)=q\ell+\min(\ell,r)$, proving the bound.
\end{proof}

We now apply this lower bound to establish
a covering property for multiset in
$\mathbb{Z}_p$ with bounded multiplicities.
\begin{lemma}
\label{lem:SigmaFull}
Let $p\ge 7$ be prime and let $U$ be a
multiset in $\mathbb Z_p$ having size $u$  and
\[
v_\alpha(U)\le h\quad\text{for all }\alpha\in\mathbb Z_p^\times,
\qquad v_0(U)=0.
\]
If $h\geq4$ and
\begin{equation}
\label{eq:mass}
u\ >\ \frac{p-2}{2}\,h,
\end{equation}
then for every integer $\ell$ satisfying
$
3\ \le\ \ell\ \le\ u-p+1
$
we have
$
\Sigma_\ell(U)=\mathbb Z_p.
$
\end{lemma}
\begin{proof}
Fix $\ell$ with $3\le \ell\le u-p+1$.
If $\Sigma_\ell(U)\neq \mathbb Z_p$,
then by Corollary~\ref{cor:DGMZp}
it suffices to show
\begin{equation}
\label{eq:Ttarget}
\sum_{\alpha\in\mathbb Z_p}\min(\ell,v_\alpha(U))\ \ge\ p+\ell-1,
\end{equation}
because then $|\Sigma_\ell(U)|\ge p$ and
hence $\Sigma_\ell(U)=\mathbb Z_p$.
Since $v_0(U)=0$, the left-hand side equals
$\sum_{\alpha\in\mathbb Z_p^\times}\min(\ell,v_\alpha(U))$.

\medskip
\noindent\textbf{Case 1: $\ell\le h$.}
Write $u=qh+r$ with $q=\lfloor u/h\rfloor$
and $0\le r<h$.
By Lemma~\ref{lem:minimize},
\[
\sum_{\alpha\in\mathbb Z_p^\times}\min(\ell,v_\alpha(U))
\ \ge\ q\ell+\min(\ell,r).
\]
From \eqref{eq:mass} we have $u/h>(p-2)/2$,
hence $q\ge (p-3)/2$.
If $q\ge (p-1)/2$, then
$
q\ell\ \ge\ \frac{p-1}{2}\,\ell.
$
For $p\ge 7$ and $\ell\ge 3$ one checks
\begin{equation}
\label{eq:ineq1}
\frac{p-1}{2}\,\ell\ \ge\ p+\ell-1,
\end{equation}
since \eqref{eq:ineq1} is equivalent to
$(\ell-2)p\ge 3\ell-2$, which holds for $p\ge 7$, $\ell\ge 3$.
Thus \eqref{eq:Ttarget} follows.
It remains to treat the only other possibility
$q=(p-3)/2$. Then
\[
\frac{u}{h}\ =\ \frac{p-3}{2}+\frac{r}{h}\ >\ \frac{p-2}{2},
\]
so $r/h>1/2$, i.e.\ $r>h/2$.
In particular $r\ge 3$ because $h\ge 4$.
If $\ell\le r$, then $\min(\ell,r)=\ell$ and
\[
q\ell+\min(\ell,r)\ =\Big(\frac{p-3}{2}+1\Big)\ell\
=\ \frac{p-1}{2}\,\ell,
\]
so \eqref{eq:Ttarget} again follows from \eqref{eq:ineq1}.
If instead $\ell>r$, then necessarily
$\ell\ge 4$ (since $r\ge 3$). Using $\min(\ell,r)=r$ and
$r>h/2\ge \ell/2$,
we obtain
\[
q\ell+r\ \ge\ \frac{p-3}{2}\,\ell+\frac{\ell}{2}\
 =\ \frac{p-2}{2}\,\ell.
\]
For $p\ge 7$ and $\ell\ge 4$, one checks
$\frac{p-2}{2}\ell\ge p+\ell-1$
(equivalently $(\ell-2)p\ge 4\ell-2$,
which holds for $p\ge 7$, $\ell\ge 4$).
So \eqref{eq:Ttarget} holds in all subcases of Case~1.

\medskip
\noindent\textbf{Case 2: $\ell>h$.}
Then $\min(\ell,v_\alpha(U))=v_\alpha(U)$
for all $\alpha$, hence
\[
\sum_{\alpha\in\mathbb Z_p}\min(\ell,v_\alpha(U))\ =
\ \sum_{\alpha\in\mathbb Z_p} v_\alpha(U)\ =\ u.
\]
Because $\ell\le u-p+1$, we have $u\ge p+\ell-1$,
i.e.\ \eqref{eq:Ttarget} holds.
Thus \eqref{eq:Ttarget} holds in all cases,
hence $\Sigma_\ell(U)=\mathbb Z_p$.
\end{proof}

With the covering property for $\mathbb{Z}_p$
established by Lemma \ref{lem:SigmaFull},
we now proceed to the main lifting result.
The following proposition combines this
projection-level covering with the structure of the
kernel to show that the $k$-fold sumset
covers the entire group $G$.
\begin{proposition}
\label{prop:lifting}
Let $p\ge 7$ be prime and let
$\pi:G\twoheadrightarrow \mathbb Z_p$ be a surjection
with kernel $H$,
$|H|=h\ge 4$.
Let $A\subseteq G$ have size $a=|A|>|G|/2$,
and let $(A_r)_{r\in\mathbb Z_p}$ be the fibers with
$a_0=\max_r a_r$.
Assume the dense fiber satisfies $\Gamma_3(A_0)=H$.
Then for every integer $k$ with
$
6\ \le\ k\ \le\ a-6,
$
we have $\Gamma_k(A)=G$.
\end{proposition}
\begin{proof}
Fix any target $\alpha\in G$.
Let
$
\overline{A}=\sum_{\beta\in A} \beta\in G
$
be the total sum of all elements of $A$.
By Lemma \ref{lem:comp},
it suffices to prove $\alpha\in \Gamma_k(A)$ for all
$6\leq k\le \lfloor a/2\rfloor$.
Fix $6\le k\le \lfloor a/2\rfloor$ and set
$
\ell=k-3.
$
We claim that $\ell\le u-p+1$,
where $u=a-a_0=|A\setminus A_0|$.
Indeed, since $a_0\le h$, we have $u\ge a-h$.
Also $a>|G|/2=ph/2$ implies $a/2>ph/4$. Because $h\ge 4$,
\[
\frac{ph}{4}-h\ =\ h\Big(\frac{p}{4}-1\Big)\ \ge\ p-4,
\]
so $ph/4\ge h+p-4$ and hence
\[
\frac{a}{2}\ >\ \frac{ph}{4}\ \ge\ h+p-4\ \ge\ a_0+p-4.
\]
Therefore
\[
\ell\ =\ k-3\ \le\ \frac{a}{2}-3\ <\ a-a_0-p+1\ =\ u-p+1.
\]
So $3\le \ell\le u-p+1$.
Let $s=\pi(\alpha)\in \mathbb Z_p$.
We now check the mass hypothesis \eqref{eq:mass}
for the multiset $U$.
Since $a_0\le h$ we have $u=a-a_0\ge a-h$, and thus
\[
u\ >\ \frac{ph}{2}-h\ =\ \frac{p-2}{2}\,h.
\]
So Lemma~\ref{lem:SigmaFull} applies and yields
$
\Sigma_\ell(U)=\mathbb Z_p.
$
Hence we can choose $\ell$  distinct  elements
$\gamma_1,\dots,\gamma_\ell\in A\setminus A_0$ with
\[
\pi(\gamma_1)+\cdots+\pi(\gamma_\ell)\ =\ s.
\]
Let $\gamma=\gamma_1+\cdots+\gamma_\ell\in G$.
Then $\pi(\gamma)=s=\pi(\alpha)$, so $\alpha-\gamma\in H$.
By the assumption $\Gamma_3(A_0)=H$, there exist
three distinct elements
$\alpha_1,\alpha_2,\alpha_3\in A_0$ such that
\[
\alpha_1+\alpha_2+\alpha_3\ =\ \alpha-\gamma.
\]
Therefore
\[
\gamma_1+\cdots+\gamma_\ell+\alpha_1+\alpha_2+\alpha_3\ =\ \alpha,
\]
a sum of $k=\ell+3$ distinct elements of $A$.
Thus $\alpha\in\Gamma_k(A)$ for all $6\leq k\le a/2$.
\end{proof}



With the lifting proposition for $p \ge 7$ established,
we can now combine it with the results for small primes
to prove our  main Theorem A for general finite abelian groups.
Although Theorem A was stated in the Introduction,
we restate it here for the convenience of the reader.

\medskip
\noindent\textbf{Theorem A.}\textit{
Let $G$ be a finite abelian group of order $g$,
and let $p(G)$ denote the smallest prime divisor of $g$.
Assume that $g$ satisfies one of the following conditions:
If $p(G)=2$, assume $g \ge 624\,
    \lvert G[2] \rvert + 1846$;
   if $p(G)=3$, assume $g \ge 3705$;
    if $p(G)=5$, assume $g \ge 6175$;
   if $p(G)\ge 7$, assume $g \ge 46319$.
\medskip
If $A \subseteq G$ satisfies $\lvert A \rvert > g/2$, then
\[
    \Gamma_k(A) = G \qquad \text{for every integer }
    k \text{ with } 3 \le k \le \lvert A \rvert - 3.
\]}
\begin{proof}
Write $a=|A|$.
By Lemma~\ref{lem:comp}, it suffices to prove
$\Gamma_k(A)=G$ for all integers
$
3\le k\le \left\lfloor \frac{a}{2}\right\rfloor.
$
The hypothesis on $g$ implies
\[
g\ \ge\ 624|G[2]|+1846\ >\ 312|G[2]|+923,
\]
and the previously established
Lemma~\ref{lem:3-4-5} yields
$
\Gamma_3(A)=\Gamma_4(A)=\Gamma_5(A)=G.
$
Fix an integer $k$ with
$6\le k\le \lfloor a/2\rfloor$.
Since $a>g/2$, we have $k\leq a/2<a-6$.

\medskip
\noindent\textbf{Case 1: $g$ is even.}
Since  $p=2$ divides $g$, there exists a surjection
$\pi:G\twoheadrightarrow \mathbb Z/2\mathbb Z$
with kernel $H$, so $[G:H]=2$ and $|H|=h=g/2$.
Since $H[2]\le G[2]$, we have $|H[2]|\le |G[2]|$.
Moreover, by our assumption,
\[
h=\frac{g}{2}\ \ge\ 312\,|H[2]|\;+\;923.
\]
Thus the hypotheses of Proposition~\ref{p=2} are satisfied, and it follows that
$\Gamma_k(A)=G$ for every $k$ with $6\le k\le a-6$.

\medskip
\noindent\textbf{Case 2: $g$ is odd.}
If $g$ is prime, then $G$ is cyclic of prime
order $g>13$, and
Corollary~\ref{cor:half_dense_prime} gives
$\Gamma_k(A)=G$ for all $3\le k\le a-3$.

Assume now that $g$ is composite.
Let $p$ be the smallest prime divisor of $g$.
Fix a surjection $\pi:G\twoheadrightarrow \mathbb Z/p\mathbb Z$
with kernel $H$, so $|H|=h=g/p$.
For each $r\in\mathbb Z/p\mathbb Z$ define fibers
\[
A_r=A\cap \pi^{-1}(r),\qquad a_r=|A_r|.
\]
 By Lemma \ref{lem:normal}, let $r_0$ maximize $a_r$ and we have $a_0>h/2$.
We split into subcases according to $p$.

\smallskip\noindent
\textbf{Subcase 2a: $p(G)=3$.}
Here $H$ has odd order $h=g/3$, and $|H[2]|=1$.
We have $h\ge 3705/3=1235$, hence $h\ge 312|H[2]|+923$.
Therefore Proposition~\ref{prop:p3}
applies and yields $\Gamma_k(A)=G$ for all $6\le k\le a-6$.

\smallskip\noindent
\textbf{Subcase 2b: $p(G)=5$.}
Similarly, $h=g/5\ge 6175/5=1235$,
so $h\ge 312|H[2]|+923$, and Proposition~\ref{prop:p5}
implies $\Gamma_k(A)=G$ for all $6\le k\le a-6$.

\smallskip\noindent
\textbf{Subcase 2c: $p(G)\ge 7$.}
We will apply Proposition~\ref{prop:lifting},
so it remains to verify its hypotheses.
First, $h\geq p\geq7>4$.
Next we prove that the dense fiber $A_0\subseteq H$
satisfies $\Gamma_3(A_0)=H$.

\smallskip\noindent
\emph{(i) If $h\ge 1235$.}
Because $a_0>h/2$, Lemma~\ref{lem:3-4-5}
applied inside the group $H$
gives $\Gamma_3(A_0)=H$.

\smallskip\noindent
\emph{(ii) If $h<1235$.}
Then   $h\le 1234$. Since $p=g/h$, $g\geq46319$ gives
\[
p=\frac{g}{h}\ >\ \frac{46319}{1234}\ >\ 37.
\]
Thus $p\ge 41$. Every prime divisor of $h$ is also
a prime divisor of $g$, and since $p$
is the smallest prime divisor of $g$,
every prime divisor of $h$ is at least $p\ge 41$.
If $h$ were composite, it would have a
prime divisor $q\ge 41$, hence $h\ge q^2\ge 41^2=1681$,
contradicting $h\le 1234$. Therefore $h$ is prime.
Consequently $H$ is cyclic of prime order $h\ge p\ge 41> 13$, and by
Corollary~\ref{cor:half_dense_prime} we have
$\Gamma_3(A_0)=H$.

\smallskip\noindent
In both cases we have shown $\Gamma_3(A_0)=H$, and therefore
Proposition~\ref{prop:lifting} applies (since $p\ge 7$, $h\ge 4$, $a>g/2$) and yields
\[
\Gamma_k(A)=G\qquad\text{for every }k\text{ with }6\le k\le a-6.
\]
\end{proof}

\subsection{Proof of Theorem B}
\label{subsec:proof_B}
We now turn to the proof of Theorem B,
which aims to pin down the  behavior of the critical number.
Recall that the $k$-critical number $\mu_k(G)$ is
defined as the smallest integer
$m$ such that any sumsubset of size $m$ covers
the entire group:
\[
\mu_k(G)=\min\bigl\{m\in\mathbb N:\ \text{every }A\subseteq G\text{ with }|A|\ge m\text{ satisfies }\Gamma_k(A)=G\bigr\}.
\]
We begin by resolving the case where $G$
has even order. In this setting,
the presence of index-$2$ subgroups creates
a sharp threshold at density $1/2$, allowing
us to determine the critical number precisely.

\begin{lemma}
\label{lem:even}
Let $G$ be a finite abelian group of   order
$g\ \ge\ 624\,|G[2]|\;+\;1846$.
If $g$ is even, then $\mu_k(G)=g/2+1$ for
$3\ \le\ k\ \le\ g/2-2$.
\end{lemma}
\begin{proof}
Since $G$ has a subgroup $H\le G$ of index $2$,
we have $|H|=g/2$.
Every restricted $k$-sum of elements of $H$ lies in $H$,
so $\Gamma_k(H)\subseteq H\neq G$.
Hence there exists a subset of size $g/2$ with
$\Gamma_k(\cdot)\neq G$, and therefore
$
\mu_k(G)\ \ge\ \frac{g}{2}+1.
$
Let $A\subseteq G$ with $|A|\ge g/2+1$.
Then $|A|>g/2$.
Consequently,
\[
k\ \le\ \frac{g}{2}-2\ \le\ |A|-3,
\]
so $k$ lies in the range required by Theorem~A.
By Theorem~A, $\Gamma_k(A)=G$.
Since this holds for every $A$ with $|A|\ge g/2+1$, we obtain
$
\mu_k(G)\ \le\ g/2+1.
$
Combining with the lower bound yields $\mu_k(G)=g/2+1$.
\end{proof}

Having determined the exact critical number for
groups of even order, we now turn our attention to the
general setting. To derive bounds for groups of odd order,
we require a structural stability result--an
inverse theorem--describing sets that fail to cover the group.
The following lemma asserts that any such set,
provided it is sufficiently dense, must be
essentially contained in a structured subset of
small index (specifically, index $2$ or $5$).

\begin{lemma}
\label{lem:B}
Let $G$ be a finite abelian group of order $g$ satisfying
$
g\ \ge\ 312\,|G[2]|+923.
$
Let $A\subseteq G$ have size $a=|A|$ and
let $k$ be an integer with $3\le k\le a$.
Set $t=k-3$ and assume
\begin{equation}\label{eq:Bsize_assumption}
a-t\ >\ \frac{5}{13}\,g.
\end{equation}
If $\Gamma_k(A)\neq G$, then \emph{one} of the
following two conclusions holds:
\begin{enumerate}[label=\textup{(\alph*)},leftmargin=2.0em]
\item there exists a subgroup $H\le G$ of
index $2$ and a coset    $C$ of $H$ such that
\[
|A\setminus C|\ \le\ t;
\]
\item there exists a subgroup $K\le G$ of index $5$ and
two cosets $C_1,C_2$ of $K$
such that
\[
|A\setminus (C_1\cup C_2)|\ \le\ t.
\]
\end{enumerate}
\end{lemma}
\begin{proof}
Assume $\Gamma_k(A)\neq G$ and choose
$\alpha\in G\setminus \Gamma_k(A)$.
Let $t=k-3$ and choose an arbitrary
subset $T\subseteq A$ with $|T|=t$.
(If $t=0$, then $T=\varnothing$.)
Set $B=A\setminus T$.
Then $|B|=a-t>5g/13$ by \eqref{eq:Bsize_assumption}.
Write $\overline{T}=\sum_{\tau\in T}\tau\in G$
(with $\overline{T}=0$ if $t=0$).

\smallskip\noindent
\textbf{Claim 1.} $\Gamma_3(B)\neq G$.

\smallskip\noindent
\emph{Proof.}
If $\Gamma_3(B)=G$, then there exist pairwise distinct
$\beta_1,\beta_2,\beta_3\in B$ such that
\[
\beta_1+\beta_2+\beta_3=\alpha-\overline{T}.
\]
Since $B\cap T=\varnothing$, the $k=t+3$
elements $T\cup\{\beta_1,\beta_2,\beta_3\}$
are pairwise distinct
and their sum equals $\alpha$,
contradicting $\alpha\notin \Gamma_k(A)$.
\hfill$\square$

\smallskip\noindent
By Claim~1 we have $\Gamma_3(B)\neq G$.
Apply Lemma~\ref{lem:Lev} to the set $B$.
Since $|B|>5g/13$, alternative \textup{(i)} fails,
and since $\Gamma_3(B)\neq G$,
alternative \textup{(iv)} fails.
Therefore either alternative \textup{(ii)} or \textup{(iii)} holds.

\smallskip\noindent
\textbf{Case 1.} Alternative \textup{(ii)} holds for $B$.
Then $B\subseteq C$ for some coset $C$ of an index-$2$ subgroup of $G$.
Since $A=B\cup T$, it follows that $A\setminus C\subseteq T$ and hence $|A\setminus C|\le |T|=t$.
This is conclusion \textup{(a)}.

\smallskip\noindent
\textbf{Case 2.} Alternative \textup{(iii)} holds for $B$.
Then $5$ is a divisor of $g$ and
$B\subseteq C_1\cup C_2$ for two cosets $C_1,C_2$ of
an index-$5$ subgroup of $G$.
Again $A=B\cup T$ implies $A\setminus (C_1\cup C_2)\subseteq T$ and hence
$|A\setminus (C_1\cup C_2)|\le t$.
This is conclusion \textup{(b)}.
\end{proof}

\smallskip\noindent
\begin{corollary}
\label{cor:odd}
Assume $G$ has odd order $g\ge1235$.
Let $A\subseteq G$ have size $a=|A|$ and let $k\in\{3,\dots,a\}$.
Set $t=k-3$.
\begin{enumerate}[label=\textup{(\alph*)},leftmargin=2.0em]
\item If $5\nmid g$ and $a-t>(5/13)g$, then $\Gamma_k(A)=G$.
\item If $5\mid g$ and $a-t>(2/5)g$, then $\Gamma_k(A)=G$.
\end{enumerate}
\end{corollary}
\begin{proof}
In odd order, $G$ has no index-$2$ subgroup,
so alternative \textup{(a)} in Lemma~\ref{lem:B} is impossible.
If $5\nmid g$, then $G$ has no index-$5$ subgroup,
so alternative \textup{(b)} is also impossible.
Thus $\Gamma_k(A)\neq G$ cannot occur under
the hypothesis $a-t>(5/13)g$, proving \textup{(a)}.

If $5\mid g$, then an index-$5$ subgroup exists,
but every union of two of its cosets has size exactly $2g/5$.
Hence $a-t>(2/5)g$ rules out alternative \textup{(b)} by size.
Again alternative \textup{(a)} is impossible in odd order,
so $\Gamma_k(A)\neq G$ cannot occur.
\end{proof}

Corollary \ref{cor:odd} highlights that the sufficient density for covering depends on the divisibility of the group order by $5$. To streamline the notation in the subsequent analysis, we encapsulate these thresholds into a density constant $c(g)$. This allows us to uniformly state the bounds for $\mu_k(G)$ for small values of $k$. Define
\begin{equation*}
c(g)=
\begin{cases}
2/5, & \text{if }5\mid g,\\[2pt]
5/13, & \text{if }5\nmid g.
\end{cases}
\end{equation*}
\begin{lemma}
\label{lem:p3_k4_k5}
Let $G$ be a finite abelian group of odd order $g$.
Let $c(g)$ be as above.
If $g\geq1235$, then
\[
\mu_4(G)\ \le\ \Bigl\lfloor c(g)g\Bigr\rfloor +2,
\qquad
\mu_5(G)\ \le\ \Bigl\lfloor c(g)g\Bigr\rfloor +3.
\]
\end{lemma}
\begin{proof}
Since   the order $g$ is odd,
Corollary~\ref{cor:odd}   applies.
Let $B\subseteq G$ satisfy $\abs{B}>\,c(g)\,g$.
Then $\Gamma_3(B)=G$ by taking $k=3$ (thus $t=0$)
in Corollary~\ref{cor:odd}.

\medskip
\noindent\textbf{The case $k=4$.}
Let $A\subseteq G$ with
$
\abs{A} \ge \bigl\lfloor c(g)\,g\bigr\rfloor +2.
$
Choose any element $\beta\in A$ and set
$A'=A\setminus\{\beta\}$. Then
$
\abs{A'}\ge \bigl\lfloor c(g)\,g\bigr\rfloor +1\ > c(g)\,g,
$
so we have $\Gamma_3(A')=G$.
Now fix an arbitrary $\alpha\in G$. Since $\Gamma_3(A')=G$,
there exist distinct  $\gamma_1,\gamma_2,\gamma_3\in A'$ such that
$
\gamma_1+\gamma_2+\gamma_3 =\alpha-\beta.
$
Then $\gamma_1,\gamma_2,\gamma_3,\beta\in A$
are pairwise distinct and $\gamma_1+\gamma_2+\gamma_3+\beta=\alpha$, hence
$\alpha\in \Gamma_4(A)$. Since $\alpha$ is arbitrary,
$\Gamma_4(A)=G$, proving
\[
\mu_4(G)\ \le\ \Bigl\lfloor c(g)\,g\Bigr\rfloor +2.
\]
The proof for the case $k=5$ is similar, and we omit its proof here.
\end{proof}

With the specific small values $k=4, 5$ handled,
we now turn to the general range of $k$.
The proof strategy relies on a case analysis based
on the smallest prime divisor $p(G)$ of the group order.
We begin with the case $p(G)=3$,
utilizing the structure of the
quotient $\mathbb{Z}/3\mathbb{Z}$.
\begin{lemma}
\label{lem:odd3}
Let $G$ be a finite abelian group of order $g$ with $p(G)=3$, and
$
g\geq3\cdot46319.
$
Let $k$ be an integer with
$
6\ \le\ k\ \le\ g/3-2.
$
Define $c(g)$ as above.
Then
\[
\mu_k(G)\ \le\ \bigl\lfloor c(g)\,g\bigr\rfloor\ +\ 9.
\]
\end{lemma}
\begin{proof}
Since $p(G)=3$,
we have a subgroup $H\le G$ of index $3$, and let $h=g/3=|H|$.
Let $A\subseteq G$ with $|A|=a\ge \lfloor c(g)g\rfloor+9$.
Let $\pi:G\twoheadrightarrow G/H\cong \mathbb{Z}/3\mathbb{Z}$
be the quotient map.
For each $r\in \mathbb{Z}/3\mathbb{Z}$ define the fibers
\[
A_r=A\cap \pi^{-1}(r),\qquad a_r=|A_r|.
\]
Let $r_0\in \mathbb Z/3\mathbb Z$ be such that $a_{r_0}=\max_{r}a_r$.
Choose any $g_0\in \pi^{-1}(r_0)$ and replace $A$ by the translate
$A'=A-g_0$.  Then $\Gamma_k(A')=G$ if and
only if $\Gamma_k(A)=G$, and with
$A'_r=A'\cap \pi^{-1}(r)$ we have $a'_0=a_{r_0}=\max_{r}a'_r$.
Hence, after translation, we may assume $a_0=\max_{r}a_r$.
By maximality of $a_0$ we have $a_0\ge \lceil a/3\rceil$.
Let
$
u=a-a_0=|A\setminus A_0|.
$
Since $g=3h$ and $5\mid g$
if and only if $5\mid h$, we also have $c(g)=c(h)$.
Since $a\ge \lfloor c(g)g\rfloor+9\geq3c(h)h+8$,
we have $a/3\geq c(h)h+8/3$ and therefore
\[
a_0\ \ge\ \left\lceil \frac{a}{3}\right\rceil\ \ge\ a/3
\ge\ c(h)h+8/3>c(h)h+2.
\]
Since $g\ge 3\cdot46319$ and $g$ is odd,
we have $h=g/3\ge46319$.
Fix $k$ with $6\ \le\ k\ \le\ \frac{g}{3}-2$ and let $k_0=\min(k,a-k)$.
By Lemma~\ref{lem:comp}, it suffices to prove $\Gamma_{k_0}(A)=G$.
We claim $k_0\ge 6$. Indeed, $k\ge 6$ by hypothesis, and also
\[
a-k\ \ge\ \lfloor c(g)g\rfloor+9\ -\ (h-2)\ \ge\
c(g)g+10 - h\ =\ (3c(h)-1)h+10.
\]
Since $c(h)\ge 5/13$ and $h\geq46319$, one has
$k_0\ge 6$. We also have $k_0\le a/2$.
Fix an arbitrary target element $\alpha\in G$.
Write $s=\pi(\alpha)\in \Z/3\Z$.

\medskip
\noindent\textbf{Case 1:  $a_0\le h/2$.}
In this case,  $u=a-a_0\ge a-h/2$.
Since $k_0\le a/2$, $a>h$ and $a_0\le h/2$, we have
\[
k_0-5\ \le\ \frac{a}{2}-5\ <
\left(a-\frac{h}{2}\right)-2\ \le\ u-2.
\]
Consider the multiset $U$ in $\Z/3\Z$ obtained by
applying $\pi$ to the elements of $A\setminus A_0$.
It has size $u$ and is supported on $\{1,2\}$.
Applying Lemma~\ref{lem:Z3} with $\ell=k_0-5$ yields
\[
\Sigma_{k_0-5}(U)\ \cup\ \Sigma_{k_0-4}(U)\ \cup\ \Sigma_{k_0-3}(U)\ =\ \Z/3\Z,
\]
so we can choose $d\in\{0,1,2\}$ and a
subset $X\subseteq A\setminus A_0$ with
\[
|X|=(k_0-5)+d\in\{k_0-5,k_0-4,k_0-3\}
\qquad\text{and}\qquad
\pi\Bigl(\sum_{\gamma\in X}\gamma\Bigr)=s.
\]
Set $m=k_0-|X|\in\{5,4,3\}$ and
$\overline{X}=\sum_{\gamma\in X}\gamma$.
Then $m\in\{3,4,5\}$ and
$
\alpha-\overline{X}\in H.
$
By $a_0>c(h)h+2$, Corollary \ref{cor:odd}
applies to $A_0\subseteq H$ and yields $\Gamma_m(A_0)=H$.
Hence there exists a subset $M\subseteq A_0$ of size $m$
with $\overline{M}=\sum_{\theta\in M}\theta=\alpha-\overline{X}$.
Because $X\subseteq A\setminus A_0$ and $M\subseteq A_0$,
the sets $X$ and $M$ are disjoint, and
$
\overline{X}+\overline{M}=\alpha.
$
Thus $\alpha\in \Gamma_{k_0}(A)$, and since $\alpha$ is arbitrary, $\Gamma_{k_0}(A)=G$.

\medskip
\noindent\textbf{Case 2: $a_0>h/2$.}
In this case, Theorem~A
applied in $H$ gives
\begin{equation}\label{eq:alllengthsB}
\Gamma_t(A_0)=H\qquad\text{for every integer }t\text{ with }3\le t\le a_0-3.
\end{equation}
Let $U$ again denote the multiset
$\pi(A\setminus A_0)$ of size $u$ supported on $\{1,2\}$.
Define
\[
j\ =\ \max\{\,0,\ k_0-(a_0-3)\,\}.
\]
Then $k_0-j\le a_0-3$ by construction,
and also $k_0-j\ge 3$ since $j\le k_0-3$.
We claim $j\le u-2$.
If $j=0$ this is immediate as $u\ge a-h\ge (3c(h)-1)h+8>2$.
If $j>0$, then $j=k_0-(a_0-3)$ and
\[
j\le u-2\ \Longleftrightarrow\ k_0-(a_0-3)\le (a-a_0)-2\
\Longleftrightarrow\ k_0+5\le a,
\]
which holds since $k_0\le a/2$ and $a>12$.
Therefore Lemma~\ref{lem:Z3} applies at length $\ell=j$,
and we can choose $d\in\{0,1,2\}$ and a subset
$X\subseteq A\setminus A_0$ with
\[
|X|=j+d\in\{j,j+1,j+2\}
\qquad\text{and}\qquad
\pi\Bigl(\sum_{\gamma\in X}\gamma\Bigr)=s.
\]
Set $m=k_0-|X|=k_0-(j+d)$.
Since $j\ge k_0-(a_0-3)$ and $d\ge 0$, we have $m\le a_0-3$.
Also, we claim $m\ge 3$. Indeed, if $j=0$, then $m=k_0-d\ge 6-2=4$;
if $j>0$, then $k_0-j=a_0-3$, so $m\ge a_0-5$,
which is $\ge 3$ since $a_0$ is large.
Thus $m\in[3,a_0-3]$, so \eqref{eq:alllengthsB}
yields $\Gamma_m(A_0)=H$.
We have $\alpha-\sum_{\gamma\in X}\gamma\in H$,
so we may choose $M\subseteq A_0$ with $|M|=m$ and
\[
\sum_{\theta\in M}\theta=\alpha-\sum_{\gamma\in X}\gamma.
\]
Then $X\cap M=\emptyset$ and $X\cup M$
is a $k_0$-subset of $A$ summing to $\alpha$.
Hence $\alpha\in \Gamma_{k_0}(A)$, and $\Gamma_{k_0}(A)=G$.
\end{proof}

Having established the critical number bound for groups
with smallest prime divisor $3$,
we now advance to the case $p(G)=5$.
While the overall strategy of utilizing the quotient structure $\mathbb{Z}/5\mathbb{Z}$ remains the same,
the arithmetic details require a slightly larger
constant term to ensure robust covering.
\begin{lemma}
\label{lem:odd5}
Let $G$ be a finite abelian group of order $g$ with smallest
prime divisor $p(G)=5$.
Assume
$
g\ \ge\ 5\cdot1235.
$
Let $k$ be any integer in the  range
$
5\ \le\ k\ \le\ \frac{g}{5}-2.
$
Then we have
 $\mu_k(G)\le 2g/5+21$.
\end{lemma}
\begin{proof}
Let $h=g/5$.
Fix $k$ with $5\le k\le h-2$.
Let $\pi:G\twoheadrightarrow \Z/5\Z$ be a surjection and let $H=\ker(\pi)$, so $|H|=h$.
The hypothesis $g\geq5\cdot1235$
implies
$
h=\frac{g}{5}\ \ge1235,
$
so Lemma~\ref{lem:Lev} applies inside $H$.
Let $A\subseteq G$ with size $a=|A|\geq2g/5+21$.
Let
\[
A_r=A\cap \pi^{-1}(r),\qquad a_r=|A_r|\qquad (r\in\Z/5\Z).
\]
Let $r_0\in \mathbb Z/5\mathbb Z$ be such that $a_{r_0}=\max_{r}a_r$.
Choose any $g_0\in \pi^{-1}(r_0)$ and replace $A$ by the translate
$A'=A-g_0$.  Then $\Gamma_k(A')=G$ if and
only if $\Gamma_k(A)=G$, and with
$A'_r=A'\cap \pi^{-1}(r)$ we have $a'_0=a_{r_0}=\max_{r}a'_r$.
Hence, after translation, we may assume $a_0=\max_{r}a_r$.
Define the outside part
\[
A_{\mathrm{out}}=A_1\cup A_2\cup A_3\cup A_4,
\qquad u=|A_{\mathrm{out}}|=|A|-a_0.
\]
Also define the multiset $U$ of residues of $A_{\mathrm{out}}$ in $\Z/5\Z$; equivalently,
$U$ contains $a_r$ copies of $r$ for each $r\in\{1,2,3,4\}$,
so the multiset $U$ has size $u$.
We will apply Corollary \ref{cor:odd} to
$A_0\subseteq H$.

\smallskip
\noindent\textbf{Case (a): $25\nmid g$.}
As  $|A|\ge 2h+21$, we have
$a_0\ge \lceil |A|/5\rceil\ge \lceil (2h+21)/5\rceil$.
Since $h\ge  1235$, one checks that
\[
\Bigl\lceil \frac{2h+21}{5}\Bigr\rceil\
>\ \frac{5h}{13}+4.
\]
Corollary \ref{cor:odd} yields
\begin{equation}\label{eq:internal_3to7}
\Gamma_m(A_0)=H\qquad\text{for every }m\in\{3,4,5,6,7\}.
\end{equation}

\smallskip
\noindent\textbf{Case (b): $25\mid g$.}
As  $|A|\ge 2h+21$, we have
$a_0\ge \lceil |A|/5\rceil\ge \lceil (2h+21)/5\rceil\ge 2h/5+5$,
so $a_0>2h/5+4$. Again
Corollary \ref{cor:odd} yields \eqref{eq:internal_3to7}.
In either case we may assume that
\eqref{eq:internal_3to7} holds.

Fix an arbitrary target $\alpha\in G$ and
set $s=\pi(\alpha)\in \Z/5\Z$.

\smallskip
\noindent\textbf{Subcase 3.1: $k\in\{5,6\}$.}
Let $t=k-3\in\{2,3\}$ and choose any $t$-subset $T\subseteq A$.
Set $B=A\setminus T$.
We have $|A|\ge 2h+21$, hence $|B|\ge 2h+18>2h=2g/5$.
Since $g$ is odd, $G$ has no index-$2$ subgroup, so alternative \textup{(ii)} in Lemma~\ref{lem:Lev} is impossible.
Also $|B|>2g/5>5g/13$ excludes \textup{(i)},
and $|B|>2g/5$ excludes \textup{(iii)}.
Therefore Lemma~\ref{lem:Lev} gives $\Gamma_3(B)=G$.
Let $\overline{T}=\sum_{\tau\in T}\tau$ .
Choose distinct $\beta_1,\beta_2,\beta_3\in B$ such that
$\beta_1+\beta_2+\beta_3=\alpha-\overline{T}$.
Then $T\cup\{\beta_1,\beta_2,\beta_3\}$ is a
$k$-subset of $A$ summing to $\alpha$.
Hence $\alpha\in\Gamma_k(A)$.

\smallskip
\noindent\textbf{Subcase 3.2: $k\ge 7$.}
Set $\ell_0=k-7\ge 0$.
We claim $\ell_0\le u-4$.
Indeed, $a_0\le h$, so $u=|A|-a_0\ge |A|-h\geq h+21>k-3=\ell_0+4$.
Therefore Lemma~\ref{lem:Z5}
applies to $U$ at $\ell=\ell_0$ and yields
an integer
\[
\ell\in\{\ell_0,\ell_0+1,\ell_0+2,\ell_0+3,\ell_0+4\}=\{k-7,k-6,k-5,k-4,k-3\}
\]
and a subset $X\subseteq A_{\mathrm{out}}$ with $|X|=\ell$ such that
\[
\pi\Bigl(\sum_{\gamma\in X}\gamma\Bigr)=s.
\]
Let $\overline{X}=\sum_{\gamma\in X}\gamma\in G$,
so $\alpha-\overline{X}\in H$.
Put $m=k-\ell$, so $m\in\{3,4,5,6,7\}$.
By \eqref{eq:internal_3to7},
$\Gamma_m(A_0)=H$, hence there exist distinct
$\alpha_1,\dots,\alpha_m\in A_0$
such that $\alpha_1+\cdots+\alpha_m=\alpha-\overline{X}$.
Since $A_0\subseteq H$ and $X\subseteq G\setminus H$, the $k$ elements in $X\cup\{\alpha_1,\dots,\alpha_m\}$ are pairwise distinct and sum to $\alpha$.
Thus $\alpha\in\Gamma_k(A)$.

\smallskip
In both subcases, $\alpha\in\Gamma_k(A)$.
Since $\alpha\in G$ is arbitrary, we have $\Gamma_k(A)=G$.
\end{proof}

Having handled the groups with smallest prime divisors
$2, 3,$ and $5$, we proceed to the general case where $p(G) \ge 7$.
The argument here relies on a different mechanism:
extending sumsets of small length to larger length $k$
by adding disjoint pairs with a constant sum.
To support this  pair padding  strategy,
the following lemma establishes a quantitative lower
bound on the number of such pairs in any sufficiently dense subset.
\begin{lemma}
\label{lem:many_pairs}
Let $G$ be a finite abelian group of odd order $g$,
and let $A\subseteq G$ have size $a=|A|$.
Then there exists $\beta\in G$ for which $A$
contains at least
\[
n_{pair}\ \ge\ \frac{1}{2}\Bigl(\frac{a^2}{g}-1\Bigr)
\]
pairwise disjoint unordered pairs
$\{\alpha_1,\alpha_2\}\subseteq A$ with
$\alpha_1\neq \alpha_2$ and $\alpha_1+\alpha_2=\beta$.
\end{lemma}

\begin{proof}
For $\beta\in G$ define
\[
n_\beta \ =\ \bigl|\{\alpha\in A:\ \beta-\alpha\in A\}\bigr|
\ =\ \sum_{\gamma\in G}\mathbf 1_A(\gamma)\mathbf 1_A(\beta-\gamma),
\]
so $n_\beta$ counts the number of
ordered  pairs $(\alpha_1,\alpha_2)\in A\times A$ with
$\alpha_1+\alpha_2=\beta$.
Summing over all $\beta$ gives
\[
\sum_{\beta\in G} n_\beta
=\sum_{\beta\in G}\sum_{\gamma\in G}\mathbf 1_A(\gamma)
\mathbf 1_A(\beta-\gamma)
=\sum_{\gamma\in G}\sum_{\beta\in G}\mathbf 1_A(\gamma)\mathbf 1_A(\beta)
=a^2.
\]
Hence $\max_{\beta}n_\beta\ge a^2/g$;
fix $\beta$ with $n_\beta\ge a^2/g$.

Now consider the involution $\alpha\mapsto \beta-\alpha$ on the set
$N_\beta=\{\alpha\in A:\ \beta-\alpha\in A\}$, which has size $|N_{\beta}|=n_\beta$.
Because $|G|$ is odd, multiplication by $2$ is a bijection on $G$,
so the equation $2x=\beta$ has at most one solution.
Thus the involution $\alpha\mapsto \beta-\alpha$ has at most
one fixed point in $N_\beta$.
Therefore $N_\beta$ can be partitioned into disjoint
2-cycles $\{\alpha,\beta-\alpha\}$, except possibly one fixed point.
Consequently $A$ contains at least $(n_\beta-1)/2$ disjoint unordered pairs $\{\alpha,\beta-\alpha\}$ with sum $\beta$.
Using $n_\beta\ge a^2/g$ gives the stated bound.
\end{proof}

Lemma \ref{lem:many_pairs} ensures that any dense
set contains a rich supply of disjoint pairs with a
common sum. We now formalize how to utilize this abundance.
The following proposition details the
pair padding  mechanism,
showing that if the number of such pairs exceeds a linear
function of $k$, we can lift the covering property from the base
cases $k=3,4$ to any larger $k$.
\begin{proposition}
\label{prop:pair_padding}
Let $G$ be a finite abelian group.
Let $A\subseteq G$ and assume $\Gamma_3(A)=\Gamma_4(A)=G$.
Fix an integer $k\ge 3$.
Assume there exists $\beta\in G$ such that $A$ contains at least
\[
n_{pair}\ \ge\ \Bigl\lfloor\frac{k-3}{2}\Bigr\rfloor +4
\]
pairwise disjoint unordered pairs
$\{\alpha_1,\alpha_2\}\subseteq A$ with $\alpha_1+\alpha_2=\beta$.
Then $\Gamma_k(A)=G$.
\end{proposition}

\begin{proof}
Fix a target $\alpha\in G$.

\smallskip
\noindent\textbf{Case 1: $k$ is odd.}
Write $k=3+2\ell$ with $\ell=(k-3)/2$.
Since $\Gamma_3(A)=G$, choose distinct
$\alpha_1,\alpha_2,\alpha_3\in A$ such that
\[
\alpha_1+\alpha_2+\alpha_3=\alpha-\ell\beta.
\]
Let $\mathcal M$ be the given family of $t$ disjoint
$\beta$--pairs in $A$.
Because the pairs in $\mathcal M$ are disjoint,
each element of $A$ belongs to  at most one  pair in $\mathcal M$.
Hence removing the three vertices
$\{\alpha_1,\alpha_2,\alpha_3\}$ eliminates at
most three pairs from $\mathcal M$.
Since $n_{pair}\ge \ell+4$, there remain at least
$\ell$ pairs in $\mathcal M$ disjoint from $\{\alpha_1,\alpha_2,\alpha_3\}$.
Pick $\ell$ such pairs, say $\{\gamma_i,\eta_i\}$ for $1\le i\le \ell$.
Then all $k$ elements $\alpha_1,\alpha_2,\alpha_3,\gamma_1,\eta_1,\dots,\gamma_\ell,\eta_\ell$
are pairwise distinct and their sum is
\[
(\alpha_1+\alpha_2+\alpha_3)+\sum_{i=1}^\ell(\gamma_i+\eta_i)=
(\alpha-\ell\beta)+\ell\beta=\alpha.
\]
Thus $\alpha\in\Gamma_k(A)$.

\smallskip
\noindent\textbf{Case 2: $k$ is even.}
Write $k=4+2\ell$ with $\ell=(k-4)/2$.
Since $\Gamma_4(A)=G$, choose distinct
$\alpha_1,\alpha_2,\alpha_3,\alpha_4\in A$ such that
\[
\alpha_1+\alpha_2+\alpha_3+\alpha_4=\alpha-\ell\beta.
\]
As above, removing $\{\alpha_1,\alpha_2,\alpha_3,\alpha_4\}$
eliminates at most four pairs from $\mathcal M$,
so $n_{pair}\ge \ell+4$ guarantees at least $\ell$ pairs disjoint from $\{\alpha_1,\alpha_2,\alpha_3,\alpha_4\}$.
Pick such $\ell$ pairs $\{\gamma_i,\eta_i\}$ and conclude as in Case~1.
\end{proof}

Finally, we apply this machinery to the general
case where the smallest prime divisor is at least $7$.
By combining the pair density estimate from Lemma
\ref{lem:many_pairs} with the lifting property of Proposition \ref{prop:pair_padding}, we demonstrate that a
density of approximately $5/13$ is sufficient to control
the critical number for   $3\leq k\leq g/p(G)-2$.
\begin{lemma}
\label{lem:Cgeq7}
Let $G$ be a finite abelian group of order
$g\ge 1235$ with smallest
prime divisor $p(G)\geq7$.
Then for every integer $k$ in the   range
\[
3\ \le\ k\ \le\ \frac{g}{p(G)}-2,
\]
the $k$-critical number satisfies the  following upper bound
\[
\mu_k(G)\ \le\ \Bigl\lfloor \frac{5g}{13}\Bigr\rfloor +3.
\]
\end{lemma}
\begin{proof}
Fix $k$ with $3\le k\le g/p(G)-2$, and let $A\subseteq G$ have size
$
a=|A|\ \ge\ \lfloor \frac{5g}{13}\rfloor+3.
$
Then $a>5g/13+1$,
hence Corollary~\ref{cor:odd} gives
$
\Gamma_3(A)=
\Gamma_4(A)=G.
$
By Lemma~\ref{lem:many_pairs},
there exists $\beta\in G$ such that $A$ contains at least
\[
n_{pair}\ \ge\ \frac{1}{2}\Bigl(\frac{a^2}{g}-1\Bigr)
\]
pairwise disjoint unordered $\beta$--pairs.
Since $a\ge 5g/13+2$, we have
\[
\frac{a^2}{g}\ \ge\ \frac{(5g/13+2)^2}{g}
= 25g/169 + 20/13 + \frac{4}{g}.
\]
Hence
\[
n_{pair}\ \ge\ \frac{1}{2}\Bigl(25g/169 + 20/13 -1\Bigr).
\]
We claim that
\begin{equation}\label{eq:pair_enough}
n_{pair}\ \ge\ \Bigl\lfloor\frac{k-3}{2}\Bigr\rfloor +4.
\end{equation}
Indeed, since $k\le g/p(G)-2$, we have
\[
\Bigl\lfloor\frac{k-3}{2}\Bigr\rfloor
\ \le\ \frac{1}{2}\Bigl(\frac{g}{p(G)}-5\Bigr),
\]
so it suffices to prove
\[
\frac{1}{2}\Bigl(25g/169 + 20/13 -1\Bigr)\ \ge\ \frac{1}{2}\Bigl(\frac{g}{p(G)}-5\Bigr)+4,
\]
equivalently
\[
\Bigl(25/169-\frac{1}{p(G)}\Bigr)g\ \ge\ 4(1-5/13).
\]
The last inequality holds for all $p(G)\geq7$.
This proves \eqref{eq:pair_enough}.
Now Proposition~\ref{prop:pair_padding} applies
(with this $\beta$ and $n_{pair}$), yielding $\Gamma_k(A)=G$.
Since $A$ was arbitrary of size at least $\lfloor 5g/13\rfloor+3$,
we conclude
$\mu_k(G)\le \lfloor 5g/13\rfloor+3$.
\end{proof}

Collecting Lemmas \ref{lem:even}, \ref{lem:p3_k4_k5},
\ref{lem:odd3}, \ref{lem:odd5} and
\ref{lem:Cgeq7}, we have arrived at Theorem B,
which is restated below.

\medskip
\noindent\textbf{Theorem B.}\textit{
Let $G$ be a finite abelian group of order $g$.
Let $k$ be an integer satisfying the range conditions associated with the smallest prime divisor $p(G)$ of $g$ such that $3 \le k \le g/p(G)-2$.
\begin{enumerate}[label=\textup{(\roman*)}, leftmargin=2em]
    \item If $g$ is even and $g\ \ge\ 624\,|G[2]|\;+\;1846$, then the $k$-critical number is exactly
    \[
    \mu_k(G) = \frac{g}{2} + 1.
    \]
    \item If $g$ is odd, define the density constant
    \[
    c(g) = \begin{cases}
    2/5 & \text{if } 5 \mid g, \\
    5/13 & \text{if } 5 \nmid g.
    \end{cases}
    \]
Then the $k$-critical number satisfies the upper bound
\[
\mu_k(G) \le \begin{cases}
  \bigl\lfloor c(g)\,g\bigr\rfloor\ +\ 9 & \text{if } p(G)=3~\text{and}~g\geq3\cdot46319, \\[2pt]
\bigl\lfloor c(g)\,g\bigr\rfloor\ +\ 21 & \text{if } p(G)=5~\text{and}~g\geq5\cdot1235, \\[2pt]
\bigl\lfloor c(g)\,g\bigr\rfloor\ +\ 3 & \text{if } p(G)\geq7~\text{and}~g\geq1235. \\[2pt]
    \end{cases}
    \]
\end{enumerate}}

\section{Proof of the Conjecture on MDS Elliptic Codes}
\label{sec:mds-elliptic}
We now apply Theorem~A to settle a conjecture of
Han and Ren~\cite{HR24} on the maximal
length of MDS elliptic codes. Throughout this section,
$q$ is a prime power, $\mathbb{F}_q$ denotes the finite
field of order $q$,
$E/\mathbb{F}_q$ is an elliptic
curve, and we write
$
G=E(\mathbb{F}_q)
$
for the finite abelian group of
$\mathbb{F}_q$-rational points;
e.g. see \cite{Sti09,Si09,HR24} for more details.
Let $P=\{P_1,\dots,P_n\}\subseteq G$ be a set of pairwise
distinct points and let $D$ be an
$\mathbb{F}_q$-rational divisor on $E$ with
$\mathrm{Supp}(D)\cap P=\varnothing$.
Write $\mathcal{L}(D)$ for the Riemann--Roch space
of $D$. The (evaluation) elliptic code associated to $(E,P,D)$ is
\[
C_\mathcal{L}(D,P)
=\Bigl\{(f(P_1),\dots,f(P_n))\in \mathbb{F}_q^n :
 f\in \mathcal{L}(D)\Bigr\}
 \subseteq \mathbb{F}_q^n.
\]
When $0<\deg(D)<n$, the evaluation map is injective and $\dim_{\mathbb{F}_q}C_\mathcal{L}(D,P)=\deg(D)$.
In this regime we set
$
k=\deg(D).
$
Recall that an $[n,k]$ linear code is called maximum distance separable
(MDS)
if it attains the Singleton bound, i.e.\ its minimum distance
is $n-k+1$; e.g. see \cite{Ball,Sti09}.
Every
$\mathbb{F}_q$-rational divisor $D$ determines
a point $Q_D\in E(\mathbb{F}_q)$; concretely, if
$D=\sum_i n_iQ_i$ then
\[
Q_D=\sum_i [n_i]Q_i,
\]
where $[m]:E\to E$ denotes multiplication by $m$.
The following lemma of Han and Ren
({\cite[Lemma~3.4]{HR24}}) reduces the MDS property
to a restricted sumset condition.

\begin{lemma}
\label{lem:HR-MDS-criterion}
Assume $0<k=\deg(D)<n$. Then $C_\mathcal{L}(D,P)$ is \emph{not}
MDS if and only if $Q_D\in \Gamma_k(P)$.
\end{lemma}
Han and Ren \cite[Theorem 1.2]{HR24} proved
that for $3<k<|P|-3$ one has
$|P|\le (|E(\mathbb{F}_q)|+5)/2$,
and they asked whether the constant term can be removed for large $q$.
More precisely, they proposed the following.

\begin{conjecture}[{\cite[Conjecture~3.7]{HR24}}]
\label{conj:37}
Let $C_\mathcal{L}(D,P)$ be an MDS elliptic code over $\mathbb{F}_q$.
If $q$ is sufficiently large and
$3\le \deg(D)\le |P|-3$, then
\[
|P|\le \frac{|E(\mathbb{F}_q)|}{2}.
\]
\end{conjecture}
We now prove Conjecture~\ref{conj:37} as a direct
application of Theorem~A. For this purpose, we need to apply
the Hasse--Weil bound, e.g. see \cite[Theorem 5.2.3]{Sti09}.

\begin{lemma}
\label{lem:Hasse}
For every elliptic curve $E/\mathbb{F}_q$ we have
\[
q+1-2\sqrt{q}\ \le\ |E(\mathbb{F}_q)|\ \le\ q+1+2\sqrt{q}.
\]
\end{lemma}

\begin{lemma}
\label{lem:q-threshold-MDS}
If $q\ge 47089$, then $|E(\mathbb{F}_q)|\ge 46656$ for every elliptic curve $E/\mathbb{F}_q$.
In particular, $|E(\mathbb{F}_q)|$ satisfies the size hypotheses of Theorem~A.
\end{lemma}
\begin{proof}
By Lemma~\ref{lem:Hasse}, we have $|E(\mathbb{F}_q)|\ge q+1-2\sqrt{q}=(\sqrt{q}-1)^2$.
For $q\ge 47089=217^2$ this lower bound is at least $(217-1)^2=216^2=46656$.
Moreover, the $2$-torsion subgroup $E(\mathbb{F}_q)[2]$ has size at most $4$.
Hence $|E(\mathbb{F}_q)|\ge 46656$ also implies
the even-order size condition in Theorem~A.
\end{proof}

\begin{theorem}
\label{thm:HanRen-conj37}
Let $q\ge 47089$ and let $E/\mathbb{F}_q$ be an
elliptic curve. Let $P\subseteq E(\mathbb{F}_q)$ and let
$D$ be an $\mathbb{F}_q$-rational divisor on $E$
with $\mathrm{Supp}(D)\cap P=\varnothing$.
If $C_\mathcal{L}(D,P)$ is MDS and
$
3\le \deg(D)\le |P|-3,
$
then
\[
|P|\le \frac{|E(\mathbb{F}_q)|}{2}.
\]
In particular, Conjecture~\ref{conj:37} holds.
\end{theorem}
\begin{proof}
Put $G=E(\mathbb{F}_q)$ and $g=|G|$.
Set $A=P\subseteq G$ and $k=\deg(D)$.
By Lemma~\ref{lem:q-threshold-MDS}, the group $G$ satisfies the hypotheses of Theorem~A.
Suppose for contradiction that $|A|>g/2$.
Since $3\le k\le |A|-3$, Theorem~A yields $\Gamma_k(A)=G$.
In particular $Q_D\in \Gamma_k(P)$.
By Lemma~\ref{lem:HR-MDS-criterion},
this implies that $C_\mathcal{L}(D,P)$ is
not MDS, a contradiction. Hence $|P|\le |E(\mathbb{F}_q)|/2$.
\end{proof}

\begin{remark}
Han and Ren note that when $|E(\mathbb{F}_q)|$ is even there exist non-trivial MDS elliptic codes of
length $|E(\mathbb{F}_q)|/2$
(see \cite[Subsection 3.4]{HR24} and the references therein),
so the bound in Theorem~\ref{thm:HanRen-conj37} is best possible in general.
\end{remark}

\end{document}